\documentclass[11pt]{amsart}

\usepackage{amsmath,amssymb,amsthm,multicol,mathtools,hyperref,dsfont,pinlabel,enumitem,dsfont}
\usepackage{comment}
\usepackage[usenames,dvipsnames]{xcolor}
\usepackage{tikz-cd}

\newcommand\cut{\setminus\!\setminus}
\newcommand\wt[1]{\scalebox{.9}{$\widetilde{
#1
}$}}
\newcommand\wh[1]{\scalebox{.9}{$\widehat{
#1
}$}}

\newcommand\F{\mathds{F}}

\newcommand\Q{\mathds{Q}}

\newcommand\bbm{\begin{bmatrix}}
\newcommand\ebm{\end{bmatrix}}

\definecolor{myPurple}{cmyk}{.5,.75,0,0}
\definecolor{myBlue}{cmyk}{.5,1,0,0}
\definecolor{myGreen}{cmyk}{1,0,1,0}
\definecolor{myGray}{cmyk}{0,0,0,.5}
\definecolor{myPink}{cmyk}{0,.73,.2,0}
\definecolor{myRed}{cmyk}{0,1,1,0}

\newcommand\White[1]{\color{white}#1\color{black}}
\newcommand\MobPos{\raisebox{-2pt}{\includegraphics[height=11pt]{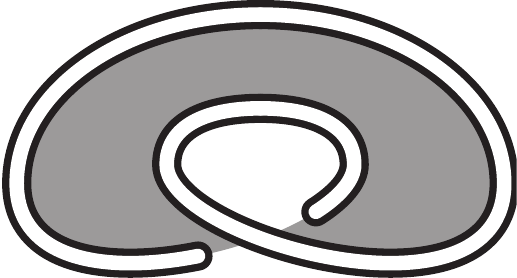}}}
\newcommand\MobNeg{\raisebox{-2pt}{\includegraphics[height=11pt]{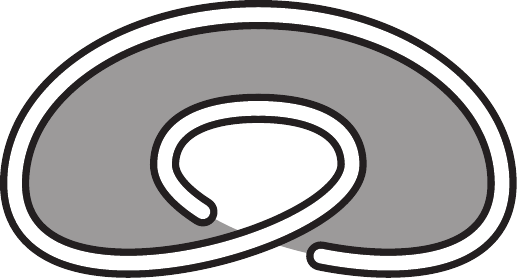}}}

\theoremstyle{plain}
\newtheorem{theorem}{Theorem}[section]

\newtheorem{prop}[theorem]{Proposition}
\newtheorem{cor}[theorem]{Corollary}
\newtheorem{conjecture}[theorem]{Conjecture}

\newtheorem*{P:GeomEss}{Proposition \ref{P:GeomEss}}
\newtheorem*{T:BadPlumb}{Theorem \ref{T:BadPlumb}}
\newtheorem*{T:PlumbEss}{Theorem \ref{T:PlumbEss}}
\newtheorem*{T:CBEss}{Theorem \ref{T:CBEss}}
\newtheorem*{C:StateEss}{Corollary \ref{C:StateEss}}
\newtheorem*{T:Singular}{Theorem \ref{T:Singular}}
\newtheorem*{T:CC}{Theorem \ref{T:CC}}
\newtheorem*{T:canon}{Theorem \ref{T:canon}}
\newtheorem*{T:DeplumbUnique}{Theorem \ref{T:DeplumbUnique}}
\newtheorem*{C:DeplumbUnique1}{Corollary \ref{C:DeplumbUnique1}}
\newtheorem*{C:DeplumbUnique2}{Corollary \ref{C:DeplumbUnique2}}
\newtheorem*{C:EssGraph}{Corollary \ref{C:EssGraph}}
\newtheorem*{T:compressible_attach_annulus}{Theorem \ref{T:compressible_attach_annulus}}
\newtheorem*{C:compressible_attach_annulus}{Corollary \ref{C:compressible_attach_annulus}}
\newtheorem*{P:Compute}{Problem \ref{P:Compute}}

\theoremstyle{definition}

\newtheorem{notation}[theorem]{Notation}
\newtheorem{definition}[theorem]{Definition}
\newtheorem{question}[theorem]{Question}
\newtheorem{problem}[theorem]{Problem}
\newtheorem{example}[theorem]{Example}
\newtheorem{rem}[theorem]{Remark}

\numberwithin{equation}{section}

\begin{document}


\title[How natural is Murasugi sum?]{How natural of a geometric operation is Murasugi sum?}

\author{Thomas Kindred}


\address{
Department of Mathematical Sciences \\ 
Burton Hall 115, Smith College \\
Northampton, MA 01063}
\email{tkindred@smith.edu}



\begin{abstract}
Gabai proved that any Murasugi sum of $\pi_1$-essential Seifert surfaces is also $\pi_1$-essential, and Ozawa extended this result to unoriented spanning surfaces.
We show, however, that the analogous statement about geometrically essential surfaces is untrue. (A spanning surface is geometrically essential if it cannot be compressed or boundary compressed to another spanning surface.)

We also ask when plumbing an unknotted annulus (with any number of twists) onto a compressible spanning surface yields a compressible surface.  We provide positive and negative examples, and we establish a simple sufficient condition. As an application, we obtain a new constructive proof of a result of Hatcher and Thurston about essential spanning surfaces for 2-bridge knots and links.
\end{abstract}

\maketitle



\section{Introduction}
Many techniques in the study of 3-manifolds $M$, and the knots and links $L\subset M$, involve analysis of the compact surfaces $F$ embedded in $M$---especially those with $\partial F=L$, which are called {\it spanning surfaces}.  Every link in every 3-manifold, however, has infinitely many distinct spanning surfaces.  In order to make the study of these surfaces more tractable, and more likely to yield pertinent information about the link or the manifold that contains it, it is common to restrict one's attention to so-called {\it essential} surfaces.

Two flavors of essential surfaces appear commonly in the literature.  One flavor involves an ``algebraic'' notion of $\pi_1$-essentiality, defined in terms of maps on fundamental groups. The other flavor involves ``geometric'' notions that can be understood in terms of surgery operations. See Definitions \ref{D:GeomEss} and \ref{D:AlgEss}.  The loop theorem implies that these notions coincide for 2-sided surfaces, including Seifert surfaces for links in $S^3$, but in general not every geometrically essential spanning surface is $\pi_1$-essential.

We find that the distinction between these two notions of essentiality is particularly salient when studying Murasugi sums and their effects on spanning surfaces in $S^3$. 
Murasugi sum, or generalized plumbing, is a way of gluing two spanning surfaces $F_0$ and $F_1$ along a disk $U$ to obtain another spanning surface $F=F_0*F_1$ (there is one extra condition---see Definition \ref{D:Plumb}). 
Gabai proved that oriented Murasugi sum respects several geometric properties of Seifert surfaces, including the (unified, in the case of Seifert surfaces)
 property of being essential, establishing the mantra, ``Murasugi sum is a natural geometric operation'' \cite{gab1,gab2}.  Ozawa proved more generally that any Murasugi sum of $\pi_1$-essential spanning surfaces (orientable or nonorientable) is also $\pi_1$-essential \cite{ozawa11}. 
%
Our main result shows that Gabai's mantra, and Ozawa's theorem, do not extend from $\pi_1$-essential surfaces to geometrically essential ones.  

 \begin{T:BadPlumb} 
 A Murasugi sum of geometrically essential spanning surfaces need not be geometrically essential. 
 \end{T:BadPlumb}

Figure \ref{Fi:BadPlumb} shows an example of this phenomenon. 

We also provide several related examples and results. In particular,
Example \ref{Ex:GabaiConverse} describes a Murasugi sum $F=F_0*F_1$ in which both $F_0$ and $F_1$ are compressible (hence inessential in both senses) but $F$ is essential (in both senses), and
Example \ref{Ex:FalseConverse} describes how a subtle modification of that example yields a surface $F$ that is, perhaps surprisingly, compressible.
 Theorem \ref{T:compressible_attach_annulus} then characterizes the subtlety captured in those two examples.
\begin{T:compressible_attach_annulus}
Suppose $D$ is a compressing disk for a spanning surface $F_1$, $\alpha$ is a properly embedded arc in $F_1$ with $|\alpha\cap \partial D|=1$, and $F_2$ is an unknotted annulus with any number of twists. Then any surface $F$ obtained by plumbing $F_2$ onto $F_1$ along $\alpha$ is compressible---this is true for algebraic and geometric compressibility.

Moreover, if $F=F_1*F_2$ as above, $\gamma$ is a core curve of $F_2$, and $\alpha'$ is a properly embedded arc in $F_2$ that is disjoint from $F_1$ with $|\alpha'\pitchfork\gamma|=1$, then $F$ has a compressing disk $X$ with $|\alpha'\pitchfork\partial X|=1$.
\end{T:compressible_attach_annulus}

See Section \ref{S:Background}, especially Definition \ref{D:AlgEss} and Example \ref{Ex:Geometrically_Incompressible}, for details regarding the terminology in Theorem \ref{T:compressible_attach_annulus}. 

As a corollary to Theorem \ref{T:compressible_attach_annulus}, we obtain a new constructive proof of a result due, essentially, to Hatcher and Thurston \cite{ht}.

\begin{C:compressible_attach_annulus}[Theorem 1(c) of \cite{ht}]
Consider a spanning surface $F=F_1*\cdots*F_k$ for a 2-bridge knot or link $L_{p/q}$ constructed from a continued fraction expansion $\frac{p}{q}=a+[b_1,\dots,b_k]$. The following are equivalent:
\begin{enumerate}[label=(\arabic*)]
\item\label{case:1} $F$ is $\pi_1$-essential. 
\item\label{case:2} $F$ is geometrically essential.
\item\label{case:3} Each $|b_i|\geq 2$.
\end{enumerate}
\end{C:compressible_attach_annulus}

In \cite{ht}, Hatcher and Thurston prove that conditions \ref{case:2} and \ref{case:3} are equivalent, using an induction argument on paths in the Farey graph, and they assert on page 226 that conditions \ref{case:1} and \ref{case:2} are equivalent.  They do not, however, provide an explicit proof of the latter equivalence, so in this sense Corollary \ref{C:compressible_attach_annulus} fills a gap in the prior literature.  Our new constructive proof is also noteworthy in light of Example \ref{Ex:432}, which shows that compressing disks for certain Hatcher-Thurston surfaces can be exceedingly complicated. We note that Corollary \ref{C:compressible_attach_annulus} is also finding immediate utility in \cite{cklsv2} and \cite{cklsv3}.

We emphasize the explicitly constructive nature of the results and techniques employed throughout.  In particular, the success that these techniques find in this paper provides motivation for the following (see \textsection\ref{S:Caps} for details):

\begin{P:Compute}
Write a computer program that, given as input a positive integer $n$ and a link diagram $P$ (described by, say, a Gauss code or a PD code), determines whether or not either checkerboard surface from $P$ has a compressing disk or $\partial$-compressing disk in which all subdisks have height at most $n$.  Then use the resulting data to catalog essential checkerboard surfaces by crossing number.
\end{P:Compute}

This paper is organized as follows.  In \textsection \ref{S:Background}, we discuss algebraic and geometric notions of essential spanning surfaces, and we review Murasugi sums, providing a thorough survey of many ways this operation has been applied in the literature. In \textsection\ref{S:Caps}, we introduce our main techniques for later proofs, and we give several warm-up examples. In \textsection\ref{S:BadPlumb}, we prove our main result, Theorem \ref{T:BadPlumb}---most of the work consists of proving the related Proposition \ref{P:GeomEss}. In \textsection\ref{S:Final}, we consider the question of when plumbing an essential annulus 
onto a compressible surface may yield an essential surface; this culminates with Theorem \ref{T:compressible_attach_annulus} and our new constructive proof of Corollary \ref{C:compressible_attach_annulus}. The Appendix provides further context and detailed discussion around our definition of geometrically essential spanning surfaces.

\section{Background}\label{S:Background}

In this section, we review spanning surfaces, algebraic and geometric notions of essential surfaces, and Murasugi sums.

\subsection{Spanning surfaces}

\begin{definition}\label{D:SpanningSurface}
A {\bf spanning surface} $F$ for a link $L$ is a compact surface, orientable or nonorientable, with no closed components which is (simultaneously) embedded with $\partial F=L$. An oriented spanning surface is called a {\bf Seifert surface}. We regard two spanning surfaces $F$ and $F'$ as equivalent, denoted $F\approx F'$, if they are {\bf freely isotopic}, meaning that there is an ambient isotopy of $S^3$ that takes $F$ to $F'$.
\end{definition}

\begin{notation}\label{N:Span}
Throughout, given a submanifold $S$ of an some ambient manifold, $\nu S$ will denote a closed regular neighborhood, $\mathring{\nu}S$ will denote the interior of this neighborhood, and $\mathring{S}$ will denote the interior of $S$.  The notations $F$ and $L$ will always denote a spanning surface in $S^3$ and its boundary.
\end{notation}

\begin{rem}\label{R:Exterior}
It is common to think of a spanning surface $F$ for a link $L\subset S^3$ as being properly embedded in the link exterior $S^3\setminus\mathring{\nu} L$, so that $\partial F$ intersects each meridian on $\partial\nu L$ transversally in one point. In this paper, however, (until the appendix) we will find it more convenient to take the perspective of Definition \ref{D:SpanningSurface}.
\end{rem}

\begin{figure}[h]
\begin{center}
{\includegraphics[height=1in]{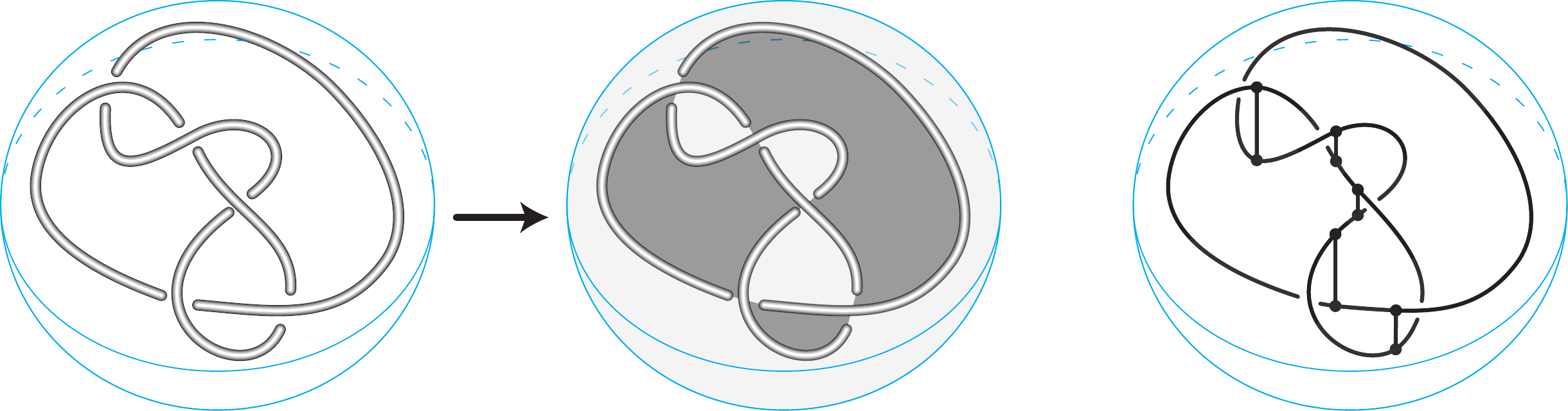}}
\caption{Constructing checkerboard surfaces}
\label{Fi:Chessboard}
\end{center}
\end{figure}

\begin{definition}
Given a diagram $P$ of $L$, one can construct two spanning surfaces $B$ and $W$ by coloring the regions of $S^2\setminus P$ black and white in checkerboard fashion. The interiors of these {\bf checkerboard surfaces} $B$ and $W$ intersect in {\it vertical arcs} which project to the crossings of $P$. 
\end{definition}

Figure \ref{Fi:Chessboard} shows this construction and the spatial graph $B\cap W$ comprised of $L$ and the vertical arcs at the crossings. The construction generalizes as follows:

\begin{definition}
Given a diagram $P$ of a link $L$, smoothing each crossing  in one of two ways, $\raisebox{-.02in}{\includegraphics[width=.125in]{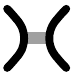}}
\overset{\color{Gray}{_{{A}}}\color{black}}{\longleftarrow}\raisebox{-.02in}{\includegraphics[width=.125in]{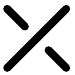}}
\overset{_{{B}}}{\longrightarrow}\raisebox{-.02in}{\includegraphics[width=.125in]{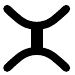}}$
yields a {\it state} $x$ of $P$, comprised of {\it state circles} and \color{Gray}$A$\color{black}- and $B$-labeled edges. One constructs an associated {\it state surface}  $F_x$ by capping off the state circles with mutually disjoint {\it state disks} (that are also disjoint from crossings and transverse to the projection sphere) and attaching a half-twisted band at each crossing\cite{ozawa11,ak}. %
The isotopy class of $F_x$ may depend on the {\it layering} of the disks relative to the projection sphere; to avoid such ambiguity, we assume, unless stated otherwise, that all state circles are capped with disks {\it lying entirely on the same side} of the projection sphere $S^2$. 
\end{definition}

\begin{figure}[h]
\begin{center}
{\includegraphics[width={4.5in}]{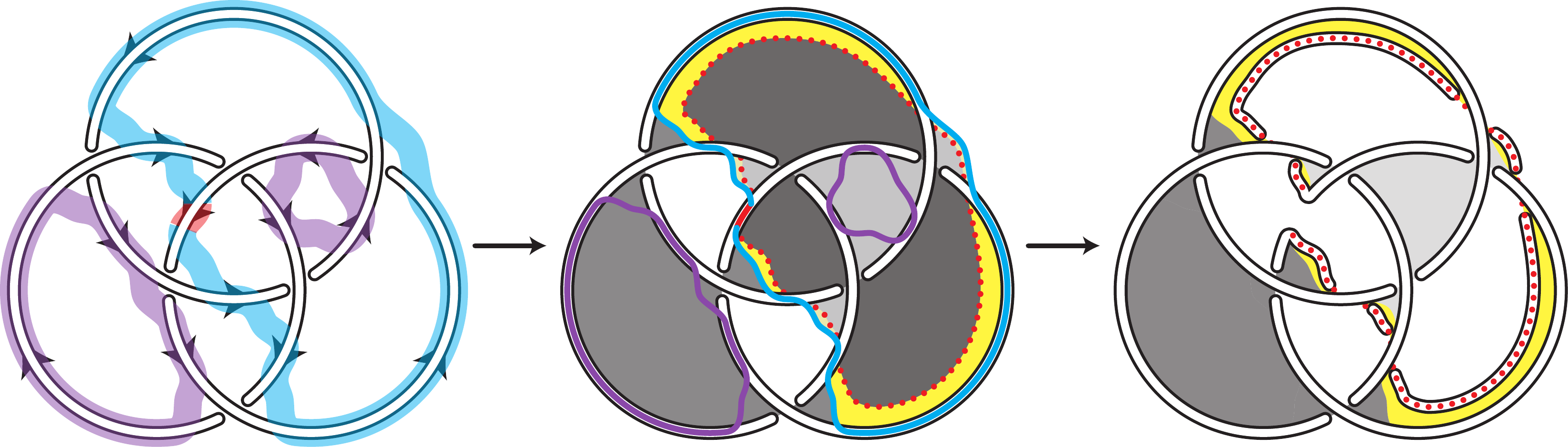}}
\caption{Realizing a Seifert surface for the Borromean rings as a checkerboard surface}
\label{Fi:StateToCBEx}
\end{center}
\end{figure}

\begin{prop}[Proposition 1.3.1 of \cite{TkThesis}]\label{P:StateToCB}
Any state surface $F_x$ from any diagram $P$ is (freely) isotopic to a checkerboard surface of some diagram $P'$.
\end{prop}

\begin{proof}
For every state circle of $x$ that is innermost on the projection sphere $S^2$, push its state disk into $S^2$.
If every state disk of $F_x$ is innermost, lying in  $S^2$, then  $F_x$ is now a checkerboard surface.  Otherwise, choose a state circle $x_0$ in $x$ that is not innermost on $S^2$, and let $U$ denote its state disk. See Figures \ref{Fi:StateToCBEx} and \ref{Fi:StateToCB}). Choose an arc $\alpha\subset x_0\cap L$ (solid red in the figures).  Take the arc $\beta=x_0\setminus\text{int}(\alpha)$ (light blue) and, fixing its endpoints, push it slightly into $\text{int}(U)$ such that its projection to $S^2$ intersects $P$ generically; call the result $\beta'$ (dotted red).  Now $\beta\cup\beta'$ bounds a bigon $U_0$ (yellow) in $U$. While fixing $\partial U$, isotope $U$ vertically so that $U_0$ becomes a union of disks in $S^2$ and half-twist bands near the crossings between $\beta'$ and $P$. Finally, isotope $\alpha$ through $U\setminus\text{int}(U_0)$ (darkest gray) to $\beta'$. The result is a state surface (for a new diagram) with one fewer non-innermost state disk than $F_x$ had.  Repeat until every state disk is innermost.
\end{proof}

\begin{figure}[h]
\begin{center}
{\includegraphics[width={4.5in}]{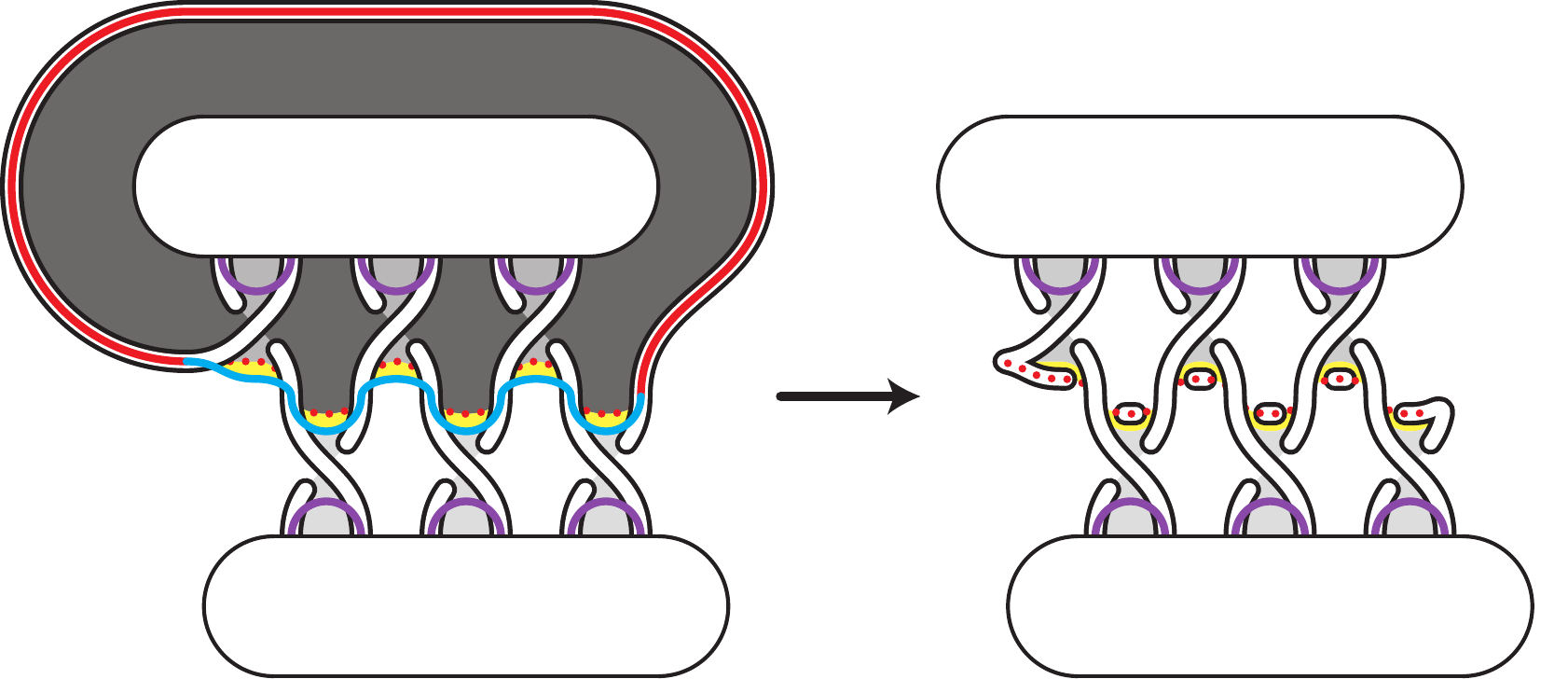}}
\caption{State surfaces are isotopic to checkerboard surfaces.}
\label{Fi:StateToCB}
\end{center}
\end{figure}

One may cut $S^3$ along a spanning surface $F$ to obtain a compact 3-manifold $S^3\cut F$ with boundary. Formally, this space is the metric closure of $S^3\setminus F$, where points are equivalence classes of Cauchy sequences that are identified if their interpolation converges. It is homeomorphic to $S^3\setminus\mathring{\nu}F$, but with extra structure from $F$ and $L$ encoded in its boundary. When $F$ is orientable, $S^3\cut F$ is a {\it sutured manifold}, and the extra structure on its boundary is a copy of $L$, which cuts $\partial(S^3\cut F)$ into two copies of $F$. When $F$ is nonorientable, however, this copy of $L$ does not separate $\partial (S^3\cut F)$, so $S^3\cut F$ is not quite a sutured manifold.  Nevertheless, this perspective will often prove useful, especially for careful statements of definitions, so we find it worthwhile to establish the following notation.   

\begin{notation}\label{N:h_F}
Given a spanning surface $F\subset S^3$, write $ h_F:S^3\cut F\to S^3$ for the quotient map that reglues corresponding pairs of points from $\mathring{F}$ in $\partial(S^3\cut F)$. Write $\wt{L}={ h_F}^{-1}(L)\subset\partial(S^3\cut F)$ and $\wt{F}= h_F^{-1}(\mathring{F})=\partial(S^3\cut F)\setminus \wt{L}$, so that  $h_F$ restricts to a homeomorphism $S^3\cut F\setminus\wt{F}\to S^3\setminus\mathring{F}$ and to a 2:1 covering map $\wt{F}\to\mathring{F}$.
\end{notation}

\subsection{Geometrically and algebraically essential surfaces}

There are two common notions of essentiality  for properly embedded surfaces in a 3-manifold: one is ``geometric,'' motivated by surgery interpretations, while the other is ``algebraic," captured by properties of the fundamental group. Both the algebraic and geometric notions of essentiality involve notions of ``incompressibility'' with ``$\partial$-incompressibility''.  When one is specifically interested in spanning surfaces for links in a 3-manifold, and not in properly embedded surfaces more generally, it makes sense to tweak the geometric notion of $\partial$-compressibility accordingly.  See \textsection1 of \cite{ht}, \textsection2 of \cite{ak}, or the Appendix of this paper for further discussion.  

\begin{figure}[h]
\begin{center}
\includegraphics[height=.75in]{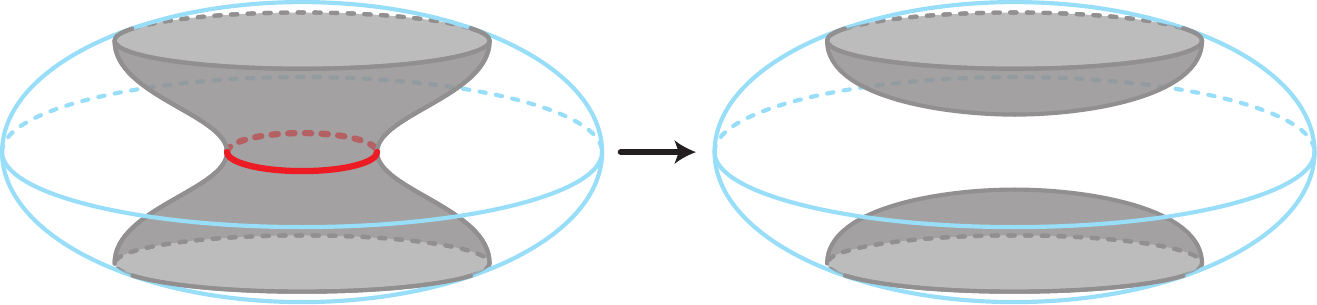}\\
\includegraphics[height=.75in]{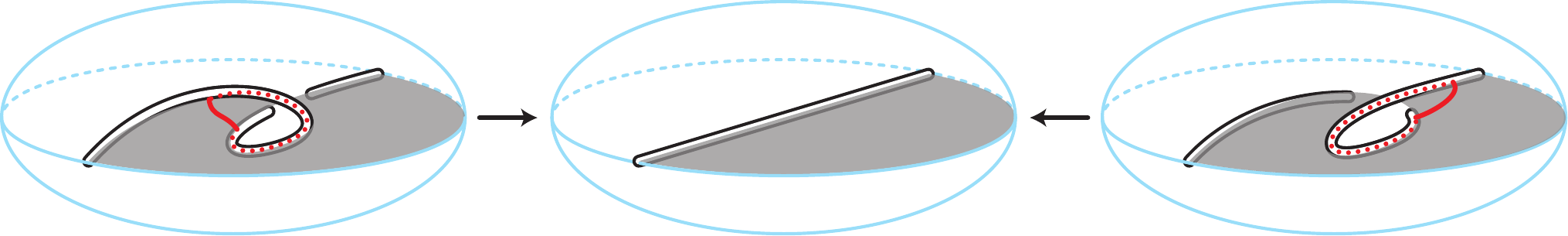}
\caption{Geometric compression (top) and $\partial$-compression (bottom) of a spanning surface}
\label{Fi:Compressions}
\end{center}
\end{figure}

\begin{definition}\label{D:GeomEss}
A surface $F$ spanning a link $L$ is {\bf geometrically essential} if $F$ cannot be compressed or ${\partial}$-compressed to a spanning surface (see Figure \ref{Fi:Compressions}).
That is, $F$ is geometrically essential if {\it both}:
\begin{enumerate}
\item For every embedded disk $D\subset S^3\setminus L$ with $D\cap F=\partial D$, the circle $\partial D$ bounds a disk in $F$, and\footnote{We use ``circle" as shorthand for ``simple closed curve". A circle in a surface is {\it essential} if it does not bound a disk in that surface.}
\item For every embedded disk $D\subset S^3\setminus L$ with $\partial D=\alpha\cup\beta$ for arcs $\alpha\subset F$ and $\beta\subset L$, the arc $\alpha$ is $\partial$-parallel in $F$.
\end{enumerate}
If $F$ satisfies (1), it is called {\bf geometrically incompressible}, whether or not it satisfies (2). Otherwise, there is an essential circle in $F$ that bounds a disk $D\subset S^3\setminus L$ with $D\cap F=\partial D$, called a (geometric) {\bf compressing disk} for $F$.
\end{definition}


\begin{rem}
If $F$ is geometrically incompressible but geometrically inessential, then $F$ is meridianally $\partial$-compressible, as defined in \cite{ak}. An equivalent condition is that $F$ is isotopic to $F'\natural\MobPos$ or $F'\natural\MobNeg$ for some spanning surface $F'$. See the Appendix.
\end{rem}

\begin{definition}\label{D:AlgEss}
A surface $F$ spanning a link $L$ is {\bf $\boldsymbol{\pi_1}$-essential} if :
\begin{enumerate}
\item Inclusion ${\text{int}(F)\hookrightarrow S^3\setminus L}$ induces an injection of fundamental groups, and 
\item $F$ is not a M\"obius band spanning the unknot, $F\not\approx\MobPos,\MobNeg$.
\end{enumerate}
If $F$ satisfies (1), it is called {\bf $\pi_1$-injective}, whether or not it satisfies (2).
\end{definition}

\begin{rem}
If $F$ is ${\pi_1}$-essential, {then} $F$ is geometrically essential.  
\end{rem}

\begin{rem}
Following \textsection1 of \cite{ht}, and assuming that $L$ is non-split, we note that $F$ is $\pi_1$-injective if and only if $\partial\nu F\cut\nu L$ is incompressible and $\partial$-incompressible in $S^3\cut\nu L$. Indeed, the loop theorem implies that $F$ is ${\pi_1}$-injective if and only if $\partial\nu F\cut\nu L$ is incompressible in $S^3\cut\nu L$. 
Further, if 
$F$ is $\partial\nu F\cut\nu L$ is incompressible  but $\partial$-compressible in $S^3\cut\nu L$, then $\partial\nu F\cut\nu L$ is a $\partial$-parallel annulus, and therefore $F=\MobPos,\MobNeg$.

It follows in particular that a Seifert surface is ${\pi_1}$-essential if and only if it is geometrically incompressible.  
\end{rem}

\begin{figure}[h]
\begin{center}
\labellist \small\hair 4pt
\pinlabel {$F_0$} at 425 46
\endlabellist
\;\hfill
{\includegraphics[width=
1.8in]{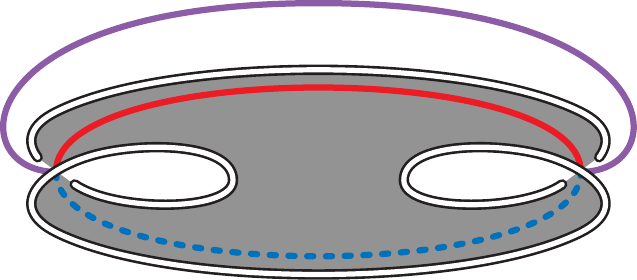}\hfill
\includegraphics[width=
1.8in
]{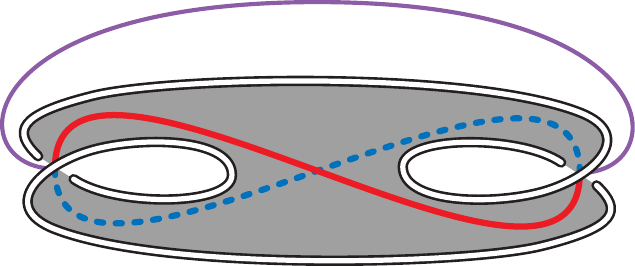}}
\hfill\;
\caption{Left: A geometrically compressible surface. Right: A geometrically incompressible surface $F_0$ which admits an algebraic compressing disk $D$: $\partial D$ is red and blue, and $\text{int}(D)\cap S^2$ is purple.}
\label{Fi:Ex1}
\end{center}
\end{figure}

\begin{example}\label{Ex:Geometrically_Incompressible}
Figure \ref{Fi:Ex1}, left, shows a surface $F_1$ which is geometrically compressible. The arcs $\alpha$ (red) and $\beta$ (blue) together bound a compressing disk $D$ which intersects the projection sphere $S^2$ in an arc $\delta$ (purple); $\delta$ thus cuts $D$ into two disks, one above $S^2$ whose boundary also includes $\alpha$ and one below $S^2$ whose boundary also includes $\beta$. Many figures in this paper use red, blue, and purple in this way; note that $\alpha$ is solid whereas $\beta$ is dashed because $D$ attaches to the front of $F$ along $\alpha$ and to the back of $F$ along $\beta$.

Figure \ref{Fi:Ex1}, right, shows a surface $F_0$ which is geometrically incompressible  but not $\pi_1$-injective. Indeed, if $F_0$ admitted a geometric compression, then the resulting surface would be a disk with the same nonzero boundary slope as $F_0$; meanwhile, the red and blue arcs on $F_0$ 
form a loop in the kernel of the induced map on fundamental groups. 

Recalling Notation \ref{N:h_F}, if one thinks of this loop as lying on $\wt{F_0}$, bounding a compressing disk $\wt{D}$ in $S^3\cut {F_0}$ for $\wt{F_0}$, then $D= h_{F_0}(\wt{D})\subset S^3$ is a disk whose interior is embedded in $S^3\setminus F_0$, and whose boundary is immersed (in this case with non-removable self-intersections). 
\end{example}

\begin{rem}
The loop theorem implies that every $\pi_1$-non-injective spanning surface $F$ admits a disk $D$ constructed in the same manner as the disk $D$ in Example \ref{Ex:Geometrically_Incompressible}. We call such $D$ an {\bf algebraic compressing disk}. It is just like a typical compressing disk, except that $\partial D$ is immersed in $\mathring{F}$ and may (or may not) self-intersect.
\end{rem}

Both examples in Figure \ref{Fi:Ex1} show that boundary sum does not respect either version of incompressibility.  In both examples, however, both boundary summands are inessential.  Later, we will modify the example shown right in the figure to make it more interesting---see Proposition \ref{P:GeomEss}.

\subsection{Murasugi sum}\label{S:Plumb}


\begin{definition}\label{D:Plumb}
Let $V\subset S^3$ be an embedded disk with $V\cap F=\partial V$ such that
\begin{enumerate}
\item $\partial V$ bounds a disk $U\subset F$.
\item Denoting $S^3\cut(U\cup V)=B_0\sqcup B_1$, neither $F_i=F\cap B_i$ is a disk.
\end{enumerate}
Then $U$ is a {{\bf plumbing disk}} for $F$, and $V$ is an associated {{\bf plumbing cap}}. 

Say that $F$ is the {\bf Murasugi sum} of $F_0$ and $F_1$ along $U$, and write $F_0*F_1=F$.  This operation 
is also called {\bf generalized plumbing}, although it is often convenient for linguistic reasons to use various forms of the phrasing ``plumbing'' more liberally, provided this does not lead to confusion. For example, we may call the associated decomposition a {\bf deplumbing} (see Figure \ref{Fi:Plumb}), and we call the operation $F\to F'=(F\setminus U)\cup V$ {\bf replumbing} (see Figure \ref{Fi:Replumb}).
\end{definition}

\begin{figure}[h]
\begin{center}
\labellist
\pinlabel{$=$} at 165 30
\pinlabel{$*$} at 325 30
\endlabellist
\includegraphics[width={4.5in}
]{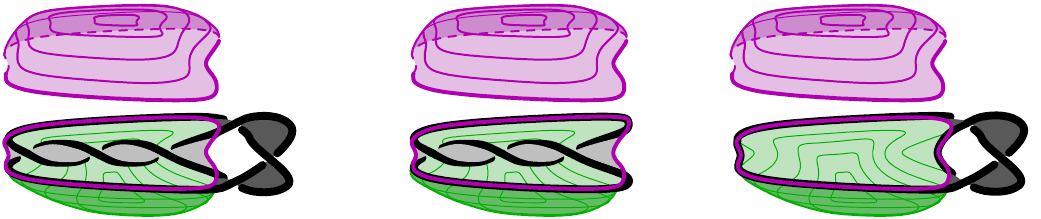}
\caption{Decomposing under Murasugi sum, or deplumbing}
\label{Fi:Plumb}
\end{center}
\end{figure}

\begin{figure}[h]
\begin{center}
\labellist
\pinlabel{$\longrightarrow$} at 185 35
\endlabellist
\includegraphics[height=.6in]{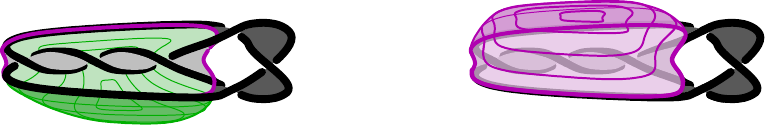}
\caption{Replumbing replaces a plumbing disk with an associated plumbing cap}
\label{Fi:Replumb}
\end{center}
\end{figure}

\begin{rem}\label{R:Boundary_Sum}
The simplest types of Murasugi sums are {\bf boundary sums} $F=F_1\natural F_2$.  Although one typically thinks of this operation as gluing $F_1$ and $F_2$ along a pair of arcs in their boundaries, one can think of it instead as gluing $F_1$ and $F_2$ along a pair of disks as suggested in Figure \ref{Fi:Boundary_Sum}.
\end{rem}

\begin{figure}[h]
\begin{center}
\labellist
\pinlabel{$\natural$} at 330 110
\pinlabel{$=$} at 660 110
\endlabellist
\includegraphics[width={4.5in}
]{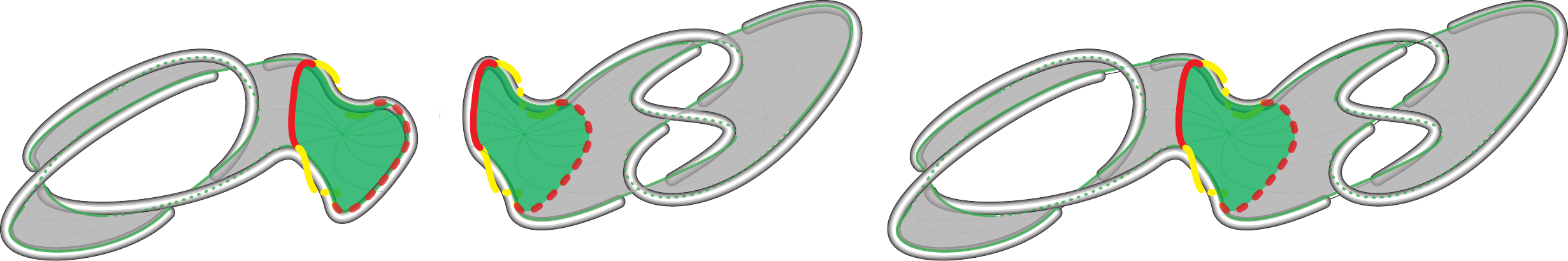}
\caption{Viewing boundary sum as Murasugi sum along a (green) bigon}
\label{Fi:Boundary_Sum}
\end{center}
\end{figure}

\begin{rem}\label{R:Plumbing}
Other than boundary sums, the simplest Murasugi sums are gluings along (topological) squares. In the context of spanning surfaces, the term ``plumbing'' (as opposed to generalized plumbing) is sometimes reserved for this specific operation.  One can always specify the surface that results from such a plumbing (up to free isotopy) by identifying the two factor surfaces $F_1$ and $F_2$, a properly embedded arc $\alpha_i$ in each surface $F_i$, a choice of local orientation on $F_i$ near $\alpha_i$, and an additional binary choice.  One can then canonically take a regular neighborhood $U_i$ of $\alpha_i$ in $F_i$; perturb $\mathring{U}_i$ in the direction specified by the orientations of $S^3$ and $F_i$ near $\alpha_i$, while fixing $\partial U_i$ and the rest of $F_i$; cut out the interior of the 3-ball swept out by the isotopy, leaving $F_i$ in a 3-ball $B_i$ with $F_i\cap\partial B_i=U_i$; and glue $B_1$ to $B_2$ by an orientation-reversing homeomorphism of their boundaries that identifies $U_1$ with $U_2$.  Indeed, there are generally two possible choices for this final gluing, which is why the specifying data included an ``additional binary choice.'' Depending on the symmetry of the factors $F_i$, however, some of this information might be unnecessary.  

In particular, if $F_2$ is an unknotted annulus or M\"obius band (and one wishes to avoid boundary sums), then one can describe a plumbing $F=F_1*F_2$ by specifying just $F_1$, $F_2$, and a properly embedded arc $\alpha\subset F_1$ and a local orientation of $F_1$ near $\alpha$.  The sufficiency of this data underpins the traditional means of describing arborescent surfaces \cite{bs}.  We will return to this topic in \textsection\ref{S:Final}.\end{rem}

A great deal is known about {\it oriented} plumbings $F=F_0*F_1$, which were first described and utilized by Murasugi in the context of knot groups \cite{mur63}. For example,  Harer showed that every fiber surface in $S^3$ can be constructed by plumbing Hopf bands and performing twisting operations introduced by Stallings \cite{harer,stallings}. Harer conjectured further that plumbing {\it and deplumbing} Hopf bands suffices, and Giroux-Goodman later proved this fact using contact topology \cite{girgoo}.  It remains an open problem to give an a more elementary proof of Harer's conjecture.  

Gabai proved that there are several geometric properties which $F$ possesses if (and usually only if) $F_0$ and $F_1$ do: 
\begin{theorem}\label{T:gabai}[\cite{gab1,gab2}]
If $F_0*F_1=F$ is a Murasugi sum of Seifert surfaces with each $\partial F_i=L_i$ and $\partial F=L$, then:
\begin{enumerate}[label=(\arabic*)]
\item\label{part:1} $F$ is {essential} if $F_0$ and $F_1$ are essential.
\item $F$ has {minimal} genus if and only if $F_0$ and $F_1$ both have minimal genus.
\item $L$ is a {fibered link with fiber} $F$ {if and only if} each $L_i$ is fibered with fiber $F_i$.
\item\label{part:4} $S^3\setminus \overset{_\circ}{\nu}L$ has a {nice codimension 1 foliation}  {if and only if} both $S^3\setminus \overset{_\circ}{\nu}L_i$ do.
\end{enumerate}
\end{theorem}

The converse of \ref{part:1} is false; more on this shortly (
see Figure \ref{Fi:GabaiConverse}).
See \cite{gab2} for details on \ref{part:4}. 

The author proved that $F$ has invertible Seifert matrix if and only if both $F_0$ and $F_1$ do and used this fact to give a simple proof of the theorem, first proven independently by Crowell and Murasugi, that the genus of an oriented alternating link equals half the breadth of its Alexander polynomial and is realized by the algorithmic Seifert surface from any alternating diagram \cite{mur58,crowell,cromur}.

Baader-Graf described a simple geometric method of fiber-detection, leading to a new proof of part (iii) of Theorem \ref{T:gabai} \cite{baadergraf}. Torisu extended (iii) to a statement about tight contact structures \cite{tor}.  Saito--Yamamoto proved that for any oriented plumbing $F=F_0*F_1$ of fiber surfaces, the arc complex for the open book decomposition of $S^3$ with page $F$ has translation distance at most two \cite{saiyam}. 
Extending (ii), Kobayashi proved that a minimal genus Seifert surface $F=F_0*F_1$ is isotopically unique if and only if $F_0$ is also unique and $F_1$ is fibered, or vice-versa \cite{kob}. Hirasawa-Sakuma used Kobayashi's result (with Menasco--Thistlethwaite's flyping theorem \cite{menthis91,menthis93,tait}) to show that certain minimal genus Seifert surfaces for alternating links cannot be constructed by applying Seifert's algorithm to an alternating diagram \cite{hs97}. Kim--Miller--Yoo showed, however, that these surfaces are all isotopic through the 4-ball \cite{kmy24}.

Oriented plumbing has also proven to be a valuable tool for studying polynomial and homological knot invariants. For example, 
Hongler--Weber \cite{hongweb04,hongweb05} used the flyping theorem to show that every oriented alternating link decomposes in a unique way under diagrammatic Murasugi sum of algorithmic Seifert surfaces from alternating diagrams, and they used this decomposition to extend results of Kobayashi--Kodama \cite{kobkod} and Murasugi--Przytycki \cite{murprz}, which also used oriented plumbing, regarding the term of the HOMFLY-PT polynomial of maximum $z$-degree. Costa--Hongler used similar techniques to study Conway polynomials of {\it periodic} alternating links \cite{coshon}. 

Perhaps the most remarkable application of oriented plumbing is Ni's plumbing-to-product formula for knot Floer homology,
\begin{equation}\label{E:HFK}
\wh{HFK}(K,g;\F)\cong\wh{HFK}(K_1,g_1;\F)\otimes\wh{HFK}(K_2,g_2;\F),
\end{equation}
where $\F$ is any field and $g,g_1,g_2$ denote 3-genus \cite{ni}.  Juh\'asz obtained a new proof of (\ref{E:HFK}) which led to a simplified proof of the fact that knot Floer homology detects fibered knots \cite{juh}. 

Rudolph constructed interesting oriented plumbings in the contexts of quasipositivity \cite{rud89} and the slice-ribbon conjecture \cite{rud02}.    

If $F_0*F_1=F$ is a plumbing of Seifert surfaces with plumbing cap $D$, then $|\partial D\cap L|=2n$ for some $n$;\footnote{Here and throughout, bars count connected components.} %
Goda established the following inequality among the handle numbers of the sutured manifolds $S^3\cut {F_i}$ and $S^3\cut F$ \cite{goda}:\footnote{The handle number $h(Y)$ of a compression body $Y$ is the minimal number of 2-handles needed to construct $Y$. The handle number of a sutured manifold $(M,\gamma)$ is $\min\{h(Y):~(Y,Y')\text{ is a Heegaard splitting of }(M,\gamma)\}$.}
\[h(S^3\cut {F_0})+h(S^3\cut{F_1})-n+1\leq h(S^3\cut {F})\leq h(S^3\cut {F_0})+h(S^3\cut {F_1}).\]
Thus, handle number is additive under boundary connect sum and is subadditive under plumbing, with defect bounded by the complexity of the plumbing.  

For any knot $K\subset S^3$ and any $s\in\Q$, let $M(K,s)$ denote the 3-manifold obtained from $S^3$ by performing Dehn surgery along $K$ with surgery slope $s$.  With $n$ as above, Li showed that $M(K,s)$ has a taut foliation for all slopes $1-n< s< n-1$ \cite{li}.

Ozbagci--Popescu-Pampu generalized the notion of Murasugi sum to {\it smooth oriented manifolds of arbitrary dimension} in such a way that part (iii) of Theorem \ref{T:gabai} extends appropriately \cite{ozbpop}. Their paper is also an excellent survey of prior literature.  A recent preprint by Karimi, Kim, Miller, Naylor, and the author considers Murasugi sums of spanning solids for knotted surfaces in $S^4$ \cite{murasugi4D}.

Perhaps the best-studied class of plumbings are the {\it arborescent surfaces}, obtained by plumbing together essential unknotted annuli and M\"obius bands according to the pattern of a tree, not just in the oriented case \cite{sak,gab86arb,kobkod} but also in the unoriented case. See the magnificent treatise by Bonahon--Siebenmann \cite{bs}.  

Unoriented plumbings appear less often in the literature than oriented ones. Recently, the author used replumbings of definite surfaces to give the first purely geometric proof of Menasco--Thistlethwaite's flyping theorem \cite{flyping,menthis91,menthis93,tait}, and to extend that result to virtual links and links in thickened surfaces \cite{virtual}. In a different paper, the author considered replumbing moves in the context of Khovanov homology \cite{khovplumb}.  

The following theorems of Ozawa conclude this survey. The first extends part (1) of Gabai's theorem to the unoriented case, and the second is a direct corollary, since checkerboard surfaces from connected reduced alternating link diagrams are $\pi_1$-essential:

\begin{theorem}
\cite[Lemma 3.4]{ozawa11}\label{T:Ozawa}
If $F=F_0*F_1$ is a Murasugi sum of $\pi_1$-essential spanning surfaces $F_i$, then $F$ is $\pi_1$-essential.
\end{theorem}

\begin{theorem}\cite[Theorem 2.8]{ozawa11}\label{T:ozawafkp}
If $x$ is a homogeneously adequate state, then the state surface $F_x$ 
is $\pi_1$-essential.
\end{theorem}

\section{Caps and height}\label{S:Caps}

Next, we introduce the technical machinery that we will need for our proofs.  The underlying ideas here are traditional.  The formalism is adapted from a more extensive treatment in the author's doctoral thesis \cite{TkThesis}, but streamlined to the specific purposes of this paper, chiefly the proofs of Theorem \ref{T:BadPlumb} and Proposition \ref{P:GeomEss}, although we will also make ample use of it in several examples.  We begin with the basic definitions, followed by an example in which the proof is just an outermost disk argument.  Then we introduce the notion of ``height,'' which increases the robustness of outermost-disk-type arguments, and we demonstrate this improved robustness with further examples.  The proof of Proposition \ref{P:GeomEss} in \textsection\ref{S:BadPlumb} will be a much more involved version of this type of argument; the examples in this section are intended largely to function as useful warm-ups for the reader, although some aspects may be of independent interest, as we discuss further in \textsection\ref{S:Final}.

\subsection{Caps and cap systems}

\begin{definition}\label{D:cap}
A {\bf cap} for $F$ is the image $V= h_F(\wt{V})$ of a compressing disk for $\partial (S^3\cut F)$. See Figure \ref{Fi:AlgCaps}. We always assume that $\partial\wt{V}\pitchfork\wt{L}$ in $\partial (S^3\cut F)$.
\end{definition}

Note that if $V$ is a cap for $F$, then we allow $\partial V$ to intersect itself or $L$. Note also that if $\partial V\cap L=\varnothing$, then $\partial V$ cannot be contractible in $F$, or else $\partial\wt{V}$ would be contractible in $\partial (S^3\cut F)$.  If $\partial V$ intersects $L$, however, then $\partial V$ may well be contractible in $F$.  For example, this is the case if $V$ is a plumbing cap.

\begin{definition}\label{D:cap_system}
A {\bf cap system} for $F$ is a union ${W}=\bigcup_i{V}_i$ of caps $V_i= h_F(\wt{V}_i)$ for $F$ with disjoint interiors, such that $\wt{W}=\bigcup_i\wt{V}_i$ cuts $S^3\cut F$ into balls, while $\partial\wt{W}$ contains $\wt{L}$ and cuts $\partial (S^3\cut F)$ into disks.  

If $W$ is a cap system for $F$ and $D$ is a cap for $F$ with $D\pitchfork W$, then the components of $D\cut W$ are called {\bf subdisks}, and for each arc $\alpha$ of $D\cap W$, each component of $D\cut \alpha$ is called a {\bf half-disk}. The same terminology applies to any disk containing a set of mutually disjoint, properly embedded arcs.
\end{definition}

\begin{notation}\label{N:CapLift}
For a cap system ${W}$, $\wt{W}$ denotes the (unique) lift which is comprised of properly embedded disks.  
\end{notation}

\begin{figure}[h]
\begin{center}
\includegraphics[width={4.5in}]{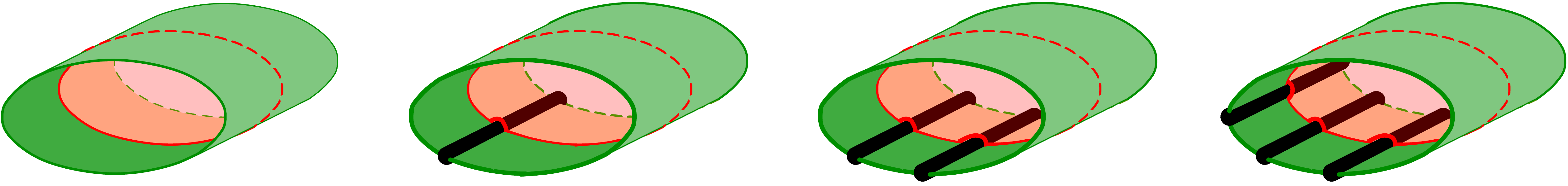}
\caption{Compressing disks for $\partial (S^3\cut F)$. Their images under $h_F$ are caps for $F$.}
\label{Fi:AlgCaps}
\end{center}
\end{figure}

\begin{example}\label{Ex:ChessboardCaps}
If $B$ and $W$ are the checkerboard surfaces from a connected link diagram, then $B$ is a cap system for $W$, and $W$ is a cap system for $B$.
\end{example}

\begin{rem}\label{R:fkp}
If ${W}$ is a cap system for $F$, then $F\cup{W}$ cuts $S^3$ into polyhedra and cuts $S^3\setminus L$ into ideal polyhedra. Futer-Kalfagianni-Purcell used such polyhedral decompositions to establish deep relationships between essential surfaces, hyperbolic geometry,  and colored Jones polynomials \cite{fkpguts,fkpquasi}.  In particular, they obtained an independent proof of Theorem \ref{T:ozawafkp} in the case that $x$ is all-$A$ or all-$B$.
\end{rem}

To extend Example \ref{Ex:ChessboardCaps} to a more general class of examples, it will sometimes be helpful to use the crossing bubbles introduced by Menasco in \cite{men84}. Given a diagram $P$ of a link $L$, insert a tiny ball $C_i$ at each crossing and perturb $P$ to get an embedding of $L$ in $(S^2\setminus C)\cup\partial C$, where $C=\bigsqcup_iC_i$. Do this so that each crossing ball $C_i$ contains precisely one straight line segment orthogonal to $S^2$ whose endpoints lie on $L$---call this a ``vertical arc'' and denote it by $c_i$; also denote $c=\bigsqcup_ic_i$.  The states of $P$ correspond to the submanifolds $x\subset (L\cup\partial C)\cap S^2$ that contain $L\cap S^2$. See Figure \ref{Fi:BubbleSmooth}. 

In this setting, $\partial C_i\cap S^2\cut L$ consists of four arcs on the equator of $\partial C_i$ for each $i$.  The union of $L\cap\partial C_i$ (the overpass and underpass at $C_i$) with either opposite pair of arcs forms a circle on $\partial C_i$, which bounds a disk in $C_i$ called a {\bf crossing band}---it is convenient to require that a crossing band in $C_i$ must contain the vertical arc $c_i$. 

\begin{figure}[h]
\begin{center}
\includegraphics[height=.75in]{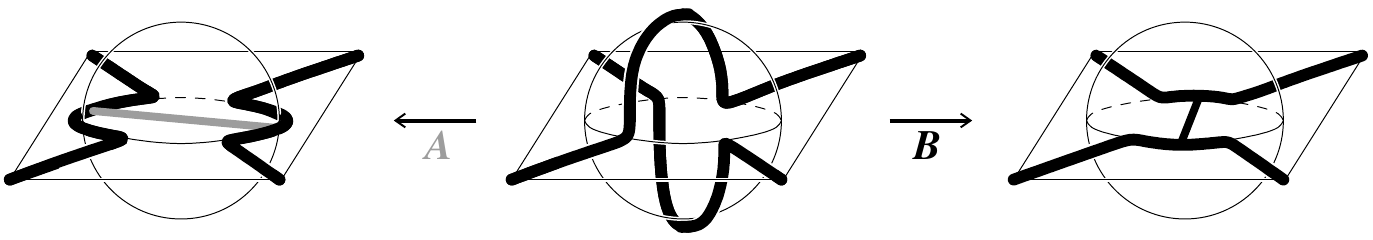}
\caption{Smoothings in the crossing ball setting}
\label{Fi:BubbleSmooth}
\end{center}
\end{figure}

\begin{definition}\label{D:FlatCapSystem}
Let $P\subset S^2$ be a diagram of a link $L$ with the type of crossing ball structure $C=\bigsqcup_iC_i$ described above, and let $F$ be a spanning surface for $L$, such that (1) $F$ intersects each $C_i$ in a crossing band, (2) $S^2\cup C$ cuts $F$ into disks, and (3) $\mathring{F}\cap(S^2\setminus C)$ is simply connected. Let ${W}_C$ be the union of the crossing bands opposite to those in $F$, isotoped so that $W_C\cap F=c$---the vertical arc $c_i$ inside each crossing bubble should cut both of the crossing bands there in half. Then $W=(S^2\cut (C\cup F))\cup{W}_y$ is a cap system for $F$, in which each cap is a disk of $S^2\cut(L\cup C\cup F)$ together with half of each abutting crossing band.  The same setup works more generally for any spanning surface $F$ for which each $F\cap C_i$ is a crossing band containing $v_i$, while $F\cut (S^2\cup C)$ and $(S^2\cut C)\cut F$ are both comprised entirely of disks.  We call $W$ the {\bf flat cap system} for (this positioning of) $F$.
\end{definition}

For example, if $F$ is a checkerboard surface, then its flat cap system is just the opposite checkerboard surface.  The cap systems we employ in this paper will all be flat cap systems. 

\subsection{First example}

In practice, capping structures ${W}$ are useful for determining, e.g., whether $F$ is $\pi_1$-essential, by helping one either find an algebraic compressing disk $D$ or prove that none exists. This works as follows.  One hypothesizes an algebraic compressing disk $D$, and assumes that, among all such disks, $D$ has been chosen to lexicographically minimize $|D\pitchfork{W}|$ and $|\partial D\cap{W}|$.\footnote{Given transverse submanifolds $S,T$ of some ambient manifold, the notations $|S\cap T|$ and $|S\pitchfork T|$ carry the same meaning; we use the latter notation if we wish to emphasize or clarify that $S$ and $T$ are transverse.}
(In some cases, it is redundant to minimize the second quantity; this is the case, for example, if every arc of $W\cap\text{int}(F)$ abuts disks of $W\cut F$ on both sides of $F$.) Then $D$ and $W$ intersect only in arcs, no circles, and no arc of $\partial D\cut W$ is parallel through $F$ to ${W}$. 
In some cases, one reaches a contradiction by finding that no outermost disk of $D\cut W$ is possible. For example:

\begin{prop}\cite[Lemma 3.3]{ozawa11}\label{P:CBEss}
If $P$ is a connected prime reduced alternating link diagram, then both checkerboard surfaces $B$ and $W$ of $P$ are $\pi_1$-essential.
\end{prop}

This is a well-known fact---in the lemma cited above, Ozawa provides three earlier references.  Alternatively, in \cite{fkpguts,fkpquasi} Futer-Kalfagianni-Purcell describe polyhedral decompositions of link complements that are closely related to the ones decompositions we describe here.  They call a decomposition {\it prime} if no pair of faces meets along more than one edge.  Their decompositions need not come from a pair of spanning surfaces, but when a prime decomposition does come from two spanning surfaces, say $B$ and $W$, it follows that $B$ and $W$ are both $\pi_1$-essential. In a nutshell, the proof is that if $B$ admitted an algebraic compressing disk $D$, then there would be an outermost disk of $D\cut W$, but primeness prohibits this. In our parlance, the proof goes like this:

\begin{proof}[Proof of Proposition \ref{P:CBEss}]
It suffices to show that $B$ is incompressible, as it can be neither \MobPos~ nor \MobNeg, and by symmetry the same argument will apply to $W$. If $B$ has a(n algebraic) compressing disk, choose one, $D$, that intersects $W$ minimally.  Now $D\cap W\neq\varnothing$ because $W$ cuts $B$ into disks (as $P$ is connected), so there is an outermost disk $D_0$ of $D\cut W$, whose boundary consists of an arc $\beta$ that lies in some disk $B_0$ of $B\cut W$ and an arc $\omega$ that lies in some disk $W_0$ of $W\cut B$.  Yet, because $P$ is prime and alternating, $B_0$ and $W_0$ meet along no more than one arc $v_0$.  Hence, $\beta$ is parallel through $B_0$ to $v_0$, contradicting minimality. 
\end{proof}

When a given polyhedral decomposition is not prime, meaning that outermost disks of $D\cut W$ are possible, Futer-Kalfagianni-Purcell describe how it is sometimes possible to refine it to produce a prime decomposition. Here, we take a different approach.  We keep the decomposition as it is and work ``upward" through the subdisks of $D\cut W$ according to the following notion of {\it height}.  

\subsection{Height}\label{S:Height}

\begin{definition}\label{D:height}
Given a disjoint union $A=\bigsqcup_{i\in I}\alpha_i$ of $\ell$ properly embedded arcs in a disk $D$, let $T$ be the tree with one vertex for each subdisk comprising $D\cut A$ in which two vertices are adjacent whenever the corresponding subdisks abut.  

Define the {\bf height} of each subdisk of $D\cut A$ recursively as follows. Let $T_0=T$. Outermost disks of $D\cut A$, corresponding to leaves in $T_0$, have height 0.   For $i\geq 1$, let $T_{i}$ be the tree obtained from $T_{i-1}$ by deleting each leaf and its incident edge.  Disks of $D\cut A$ that correspond to leaves in $T_i$ have height $i$.   Figure \ref{Fi:TreeLabel} shows an example.

For each half-disk in $D$ (as cut by $A$) the height of the half-disk is the largest height among the subdisks it contains.
\end{definition}


\begin{prop}
Let $A=\bigsqcup_{i\in I}\alpha_i$ be a disjoint union of properly embedded arcs in a disk $D$, and let $D_0$ be a subdisk of $D\cut A$, say of height $n$.  Then either every subdisk of $D\cut A$ has height at most $n$, or there is exactly one half-disk in $D$ of height $n$ that contains $D_0$.
\end{prop}

\begin{proof}
There is at most one subdisk of $D\cut A$ incident to $D_0$ whose height exceeds $n$.  If there is such a subdisk, then there is only one half-disk of $D$ that contains $D_0$ but not this other subdisk.  If there is no such subdisk, then every subdisk of $D\cut A$ has height at most $n$.
\end{proof}

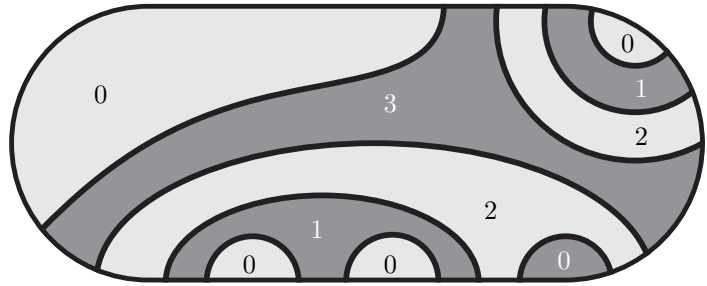
\begin{figure}[h]
\begin{center}
\labellist \small \hair 4pt
\pinlabel {$0$} at 50 100
\pinlabel {$\White{3}$} at 200 95
\pinlabel {$2$} at 252 40
\pinlabel {${2}$} at 330 78
\pinlabel {$\White{1}$} at 162 30
\pinlabel {$0$} at 127 12
\pinlabel {$0$} at 200 12
\pinlabel {$\White{0}$} at 290 14
\pinlabel {$\White{1}$} at 330 103
\pinlabel {$0$} at 323 126
%
\endlabellist
\includegraphics[height=1.5in]{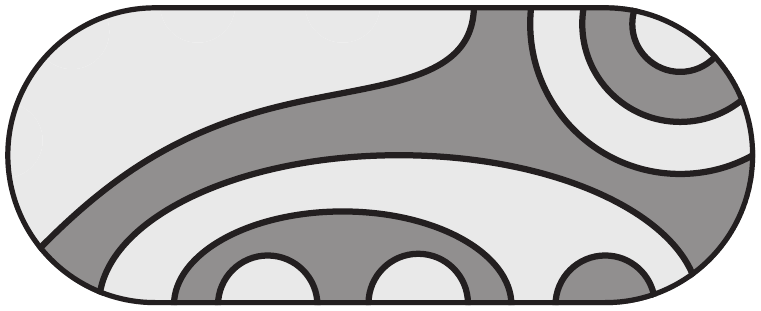}\hfill
\caption{A disk $D$ cut by arcs and labeled by height}
\label{Fi:TreeLabel}
\end{center}
\end{figure}
 


\subsection{Further examples}

In general, given a cap system $W$ for a surface $F$, one can try to determine whether or not $F$ is $\pi_1$-essential or geometrically essential by hypothesizing an appropriate compressing disk or $\partial$-compressing disk $D$ which lexicographically minimizes $|D\pitchfork{W}|$ and $|\partial D\cap W|$ and characterizing the possibilities for subdisks of $D\cut W$ of height 0, then for those of height 1, and so on, working ``upward.'' (When looking for a subdisk of height 1, e.g., we are really thinking about the half-disk of height 1 that contains it---this allows us to utilize the information we have already gathered about subdisks of height 0.)
This approach is useful from a problem-solving perspective, as the process may terminate either with a contradiction or with the successful construction of a compressing disk---the process is unsuccessful only if the cases proliferate too much. 

Several examples follow. In each, $F$ is some spanning surface, and $W$ is a flat cap system for $F$ as in Definition \ref{D:FlatCapSystem}, $D$ is a hypothesized algebraic compressing disk, which we assume has been isotoped to minimize $|D\cap W|$, and we characterize the possible subdisks of $D\cut W$ according to height.  The motor of the argument in each example is that each arc of $\partial D\cap C$ must contain an endpoint of an arc of $D\cap W$, and no arc of $\partial D\cut C$ has both endpoints on the same crossing ball (due to the minimality of $|D\cap W|$). We keep our arguments terse by allowing this motor to run silently.

Heuristically, the reader can sidestep the formality of cap systems by treating a cap system interchangeably with the projection sphere $S^2$ itself.  Technically, this is not quite correct, because we need to be more careful than this near crossings (and, sometimes, other places that $\text{int}(F)$ intersects $S^2$).  Still, it can be useful.

In Figures \ref{Fi:Pretzel222333}-\ref{Fi:GabaiConverse}, which accompany these examples, we will shade possible height 0 (outermost) subdisks of $D\cut W$ red or blue, according to whether they are on the near or far side of the projection sphere $S^2$. For subdisks of $D\cut W$ of positive height, however, shading would obscure too much of the figure, so instead we typically show just $\partial D$ (in red and blue according to whether the incident component of $D\cut W$ is on the near or far side of $S^2$, solid or dotted according to front or back) and $D\cap W$ (in purple). The same applies to Figures \ref{Fi:BadPlumb}-\ref{Fi:GeomEssStep6}.

\begin{figure}[h]
\begin{center}
\;\hfill\includegraphics[height=1.5in]{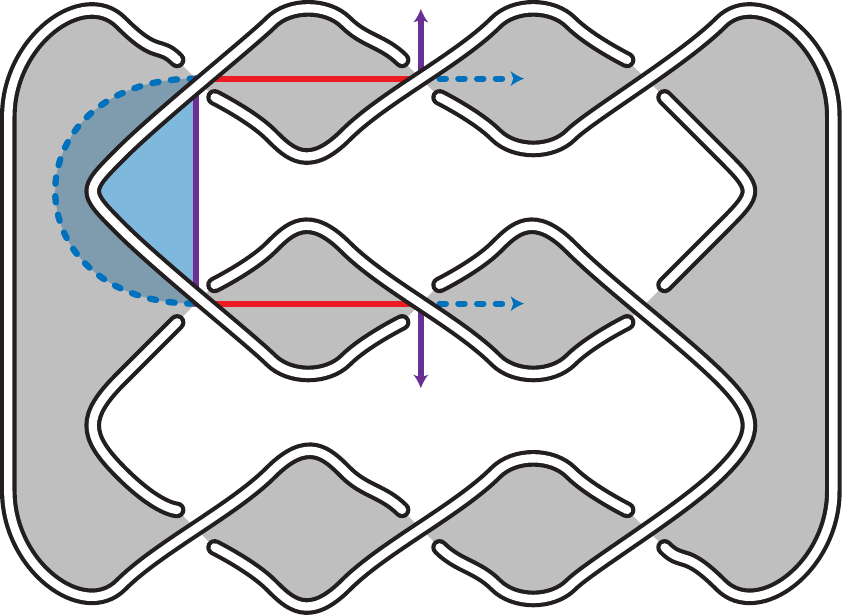}\hfill
\includegraphics[height=1.5in]{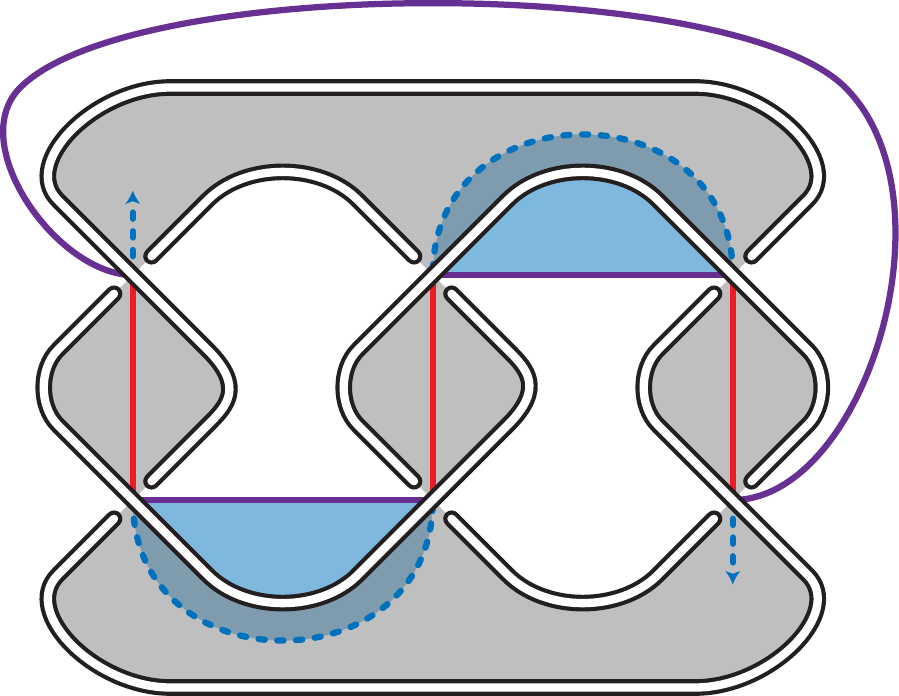}\hfill\;
\caption{$\pi_1$-essential checkerboard surfaces for the $(-3,3,-3)$ and $(-2,2,-2)$ pretzel links}
\label{Fi:Pretzel222333}
\end{center}
\end{figure}

\begin{example}
The $(-4,4,-4)$ pretzel surface $F$ shown left in Figure \ref{Fi:Pretzel222333} and its flat cap system admit subdisks of height 0--one is shown, and the others are the same up to symmetry (two of them on the other side of $S^2$). Yet, they admit no subdisk of height 1, so $F$ is $\pi_1$-essential.
\end{example}

\begin{example}\label{Ex:222}
The $(-2,2,-2)$ pretzel surface $F$ shown right in Figure \ref{Fi:Pretzel222333} and its flat cap system admit subdisks of height 0 and 1 (shown, up to symmetry), but not of height 2, so $F$ is $\pi_1$-essential. 
\end{example}

\begin{rem}
In Example \ref{Ex:222}, if we were interested only in proving that $F$ is $\pi_1$-essential (rather than building familiarity with our techniques), we could have provided either of the following proofs:\begin{enumerate} 
\item Compressing $F$ would yield a disjoint union of a disk and an annulus, but this is impossible since each pair of link components has nonzero linking number.
\item $F$ can be constructed by performing a Stallings twist on a boundary connect sum of two Hopf bands of opposite signs. Therefore, $F$ is a fiber surface, hence incompressible.
\end{enumerate}
\end{rem}

\begin{figure}[h]
\centering
 \labellist\small
 \pinlabel {$F_0$} at 190 65
 \pinlabel {$F_1$} at 492 5
 \pinlabel {$F$} at 790 65 
 \endlabellist
\includegraphics[width=
2.5in]{figures/GabaiConverseB11}\hfill\raisebox{.5in}{$*$}\hfill
\raisebox{.125in}{\includegraphics[width=
1in]{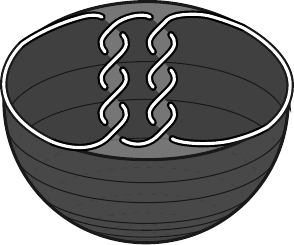}}\hfill\raisebox{.5in}{$=$}\hfill
\includegraphics[width=
2.5in]{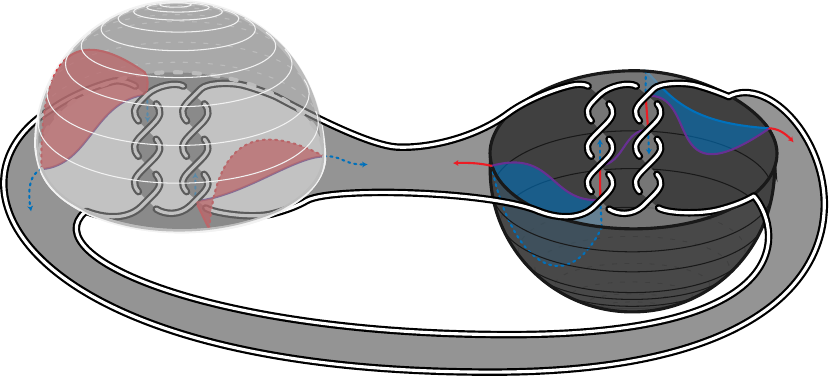}
\caption{An essential Seifert surface $F$ obtained by plumbing an essential surface onto a compressible one.}
\label{Fi:GabaiConverse1}
\end{figure}

\begin{example}\label{Ex:GabaiConverse}
Consider the surface $F=F_0*F_1$ constructed in Figure  \ref{Fi:GabaiConverse} by plumbing an essential pair of pants $F_1$ onto a compressible surface $F_0$ (obtained by plumbing another such pair of pants onto a trivial annulus).  The flat cap system $W$ for $F$ admits subdisks of height 0 (shown), but none of height 1. Therefore, $F$ is indeed $\pi_1$-essential.
\end{example}

Example \ref{Ex:GabaiConverse} demonstrates that it is possible to plumb an essential Seifert surface onto a compressible one in a way that yields an essential surface.  Gabai gave a similar example of this phenomenon in Figure 2 of \cite{gab1}.  The next example shows that it is in fact possible to plumb two Seifert surfaces, {\it both} of them compressible, in a way that yields an essential surface. To the author's knowledge, it is the first such example in the literature.

\begin{example}\label{Ex:StrongConverse}
Consider the surface $F$ shown right in Figure \ref{Fi:GabaiConverse}, obtained by plumbing the two compressible Seifert surfaces $F_1,F_2$ shown left in each figure.   (Figure \ref{Fi:GabaiConverseHeight1} may help the reader visualize $F_2$.)  We claim that $F$ is essential. Indeed, with the flat cap system from Figure \ref{Fi:GabaiConverse}, it admits four subdisks of height 0, as shown in the figure. It also admits a subdisk $D$ of height 1, as shown in Figure \ref{Fi:GabaiConverseHeight1}. These, however, are the only subdisks of any heights that it admits.   
\end{example}

\begin{figure}[h]
 \begin{center}
\includegraphics[width=
2.5in]{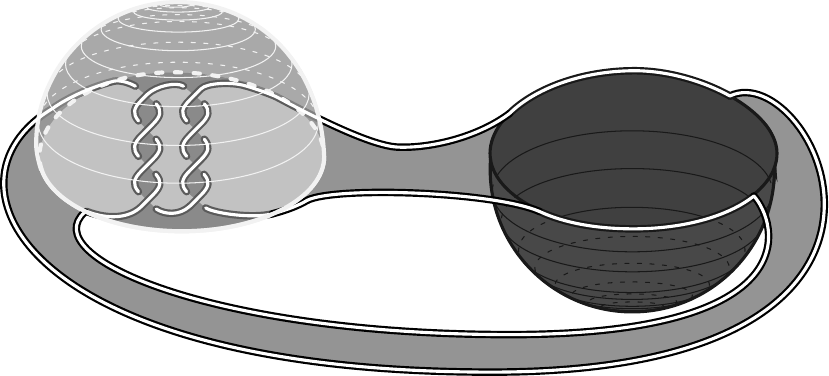}\hfill\raisebox{.5in}{$*$}\hfill
\raisebox{.125in}{\includegraphics[width=
1in]{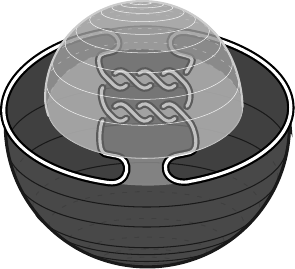}}\hfill\raisebox{.5in}{$=$}\hfill
\includegraphics[width=
2.5in]{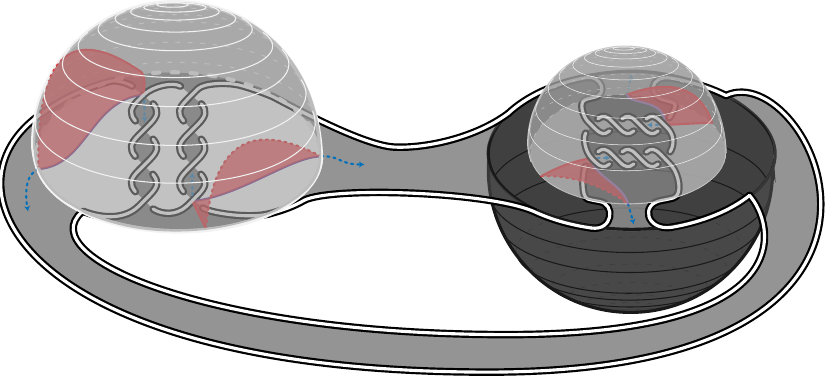}
\caption{An essential Seifert surface $F$ obtained by plumbing together two compressible surfaces}\label{Fi:GabaiConverse}
\end{center}
\end{figure}

\begin{figure}[h]
 \begin{center}
\includegraphics[width={4.5in}]{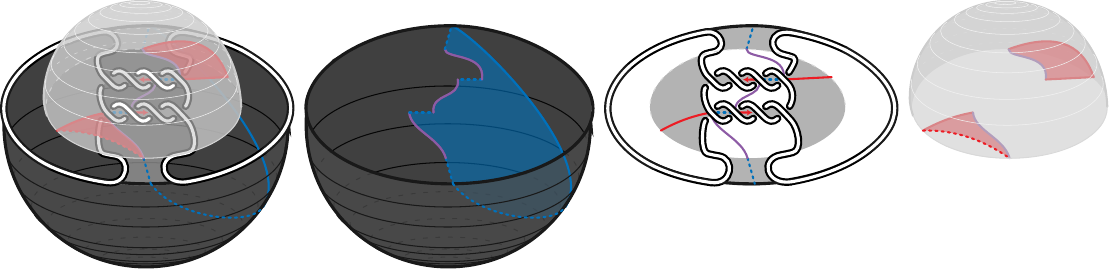}
\caption{Far left: the height 1 subdisk $D$ for the surface $F=F_1*F_2$ shown right in Figure \ref{Fi:GabaiConverse}. Center and right: the parts of $F_2\subset F$ and $D$ below, near, and above $S^2$.}
\label{Fi:GabaiConverseHeight1}
\end{center}
\end{figure}

The two preceding examples motivate several natural questions, to which we will return in \textsection\ref{S:Final}.  For now, however, we proceed to our main result.

\section{Geometric essentiality under unoriented plumbing}\label{S:BadPlumb}

\begin{figure}[h]
\begin{center}
\labellist\small\hair 4pt
\pinlabel plumbing at 800 145
\endlabellist
\includegraphics[width={6.5in}]{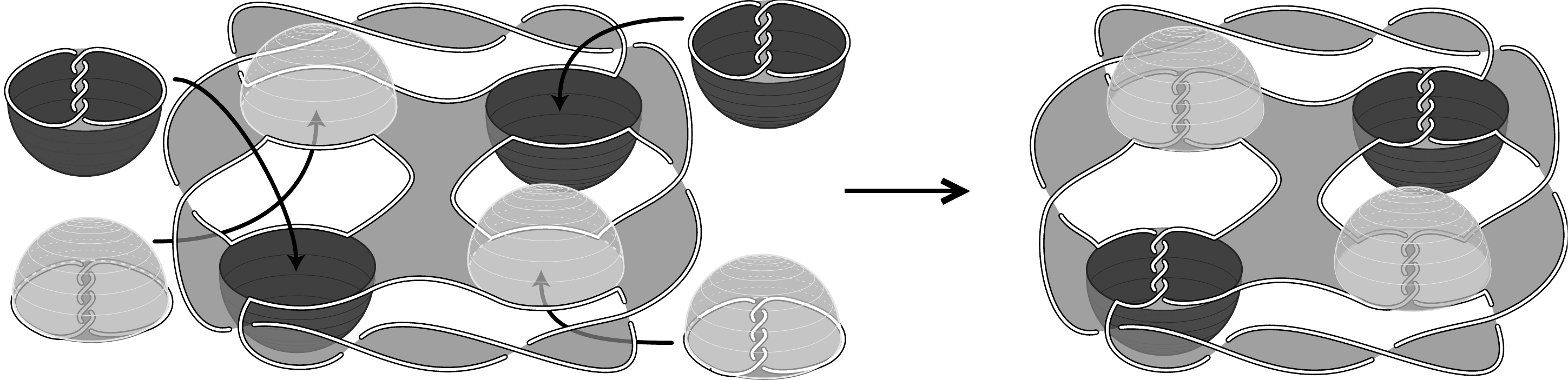}
\caption{Constructing a surface $F_1$ that is geometrically essential but $\pi_1$-inessential}
\label{Fi:Ex2}
\end{center}
\end{figure}

In this section, we will use cap systems and height to prove that the surface $F_1$ constructed in Figure \ref{Fi:Ex2} is geometrically essential, giving our first main result:

 \begin{theorem}\label{T:BadPlumb}
 A Murasugi sum of geometrically essential surfaces need not be geometrically essential.
 \end{theorem}
 
 \begin{proof}
 By plumbing a Hopf band onto the surface in Figure \ref{Fi:Ex2} as shown in Figure \ref{Fi:BadPlumb}, one can obtain a geometrically compressible surface. Indeed, the boundary of a compressing disk and its intersection with the projection sphere are colored in the rightmost part of the figure. The theorem now follows the following proposition.
\end{proof}
 
 \begin{figure}[h]
 \begin{center}
  \labellist
\small\hair 4pt
\pinlabel {${*}$} [l] at 475	 150
\pinlabel {${=}$} [l] at 640 150
\endlabellist
\includegraphics[width={6.5in}]{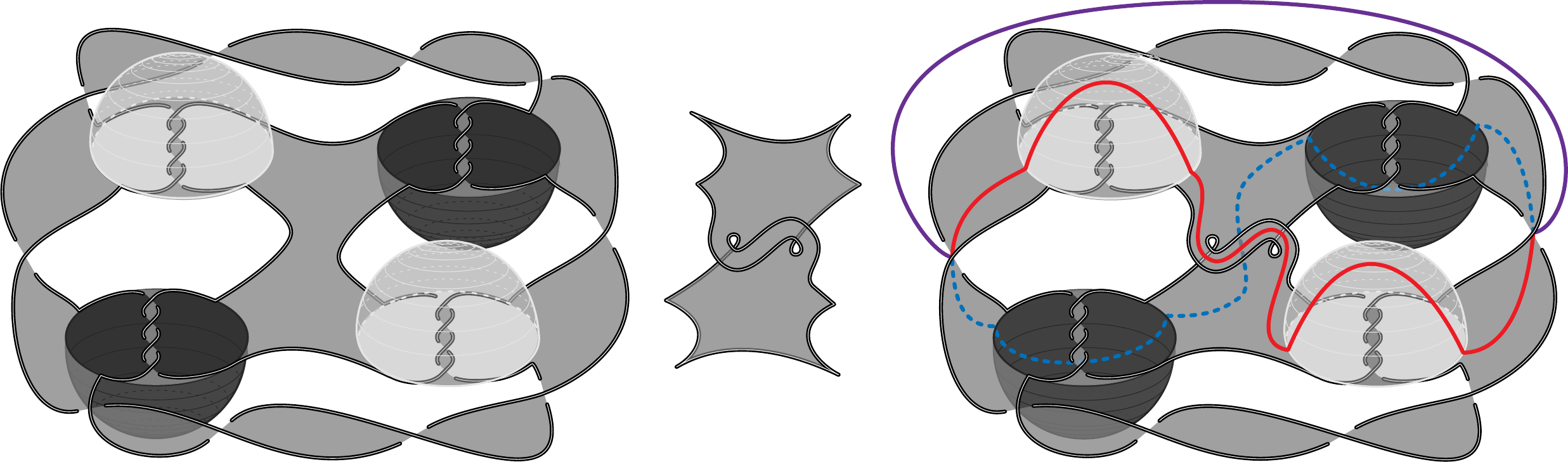}
\caption{A geometrically inessential plumbing of geometrically essential surfaces}
\label{Fi:BadPlumb}
\end{center}
\end{figure}

\begin{prop}\label{P:GeomEss}
The surface $F_1$ constructed in Figure \ref{Fi:Ex2} is geometrically essential but $\pi_1$-inessential.
\end{prop}

\begin{proof}
Certainly $F_1$ is algebraically compressible (i.e. $\pi_1$-non-injective), as the plumbed-on annuli do not obstruct the algebraic compressing disk from Figure \ref{Fi:Ex1}.  To see that $F_1$ is geometrically essential, isotope $F_1$ to appear as the checkerboard surface shown in Figure \ref{Fi:GeomEssStep1} (again using the method described in the proof of Proposition \ref{P:StateToCB}),
and consider the resulting flat cap system ${W}$.  Write $\text{int}(F_1)\cap{W}=v$, so that $F_1\cap {W}=L\cup v$; $v$ consists of vertical arcs, one at each crossing.

Suppose first that $F_1$ is (geometrically) compressible.  Choose a compressing disk $D$ for $F_1$ which minimizes $|D\pitchfork{W}|$.  Then $D\cap{W}$ consists entirely of arcs, each with endpoints on distinct arcs of $v$. Likewise, each arc of $\partial D\cut v$ has endpoints on distinct arcs of $v$, and each point of $\partial D \cap v$ is an endpoint of an arc of $D\cap W$ and of two arcs of $\partial D\cut v$.  

Now consider the possibilities for the subdisks of $D\cut{W}$. Start with those of height 0, each of whose boundary consists of just two arcs, one in $D\cap W\subset {W}\cut v$ and one in $\partial D\cut v\subset F\cut v$.
Up to symmetry, there are three types of height 0 subdisks, two of each type in each ball of $S^3\cut(F_1\cup{W})$; all twelve possible subdisks appear in Figure \ref{Fi:GeomEssStep1}.

\begin{figure}[h]
\begin{center}
\includegraphics[width=
3.6in]{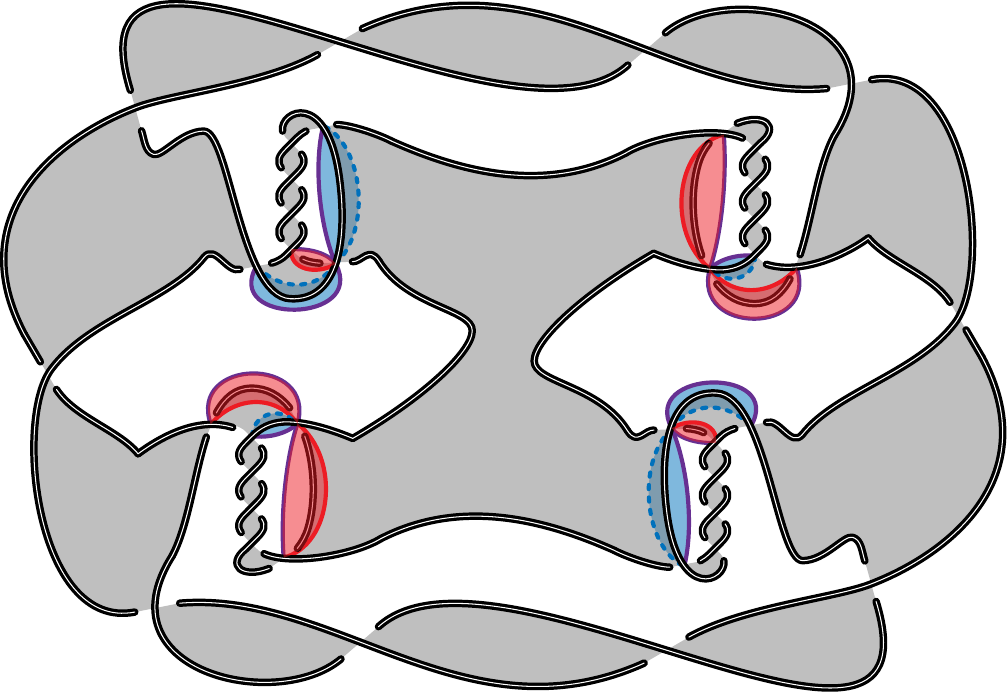}
\caption{Possible height 0 subdisks in the proof of Proposition \ref{P:GeomEss}}
\label{Fi:GeomEssStep1}
\end{center}
\end{figure}

There are three types of height 1 subdisks (up to symmetry) that abut a single outermost subdisk, having pattern $\raisebox{-2pt}{\includegraphics[height=12pt]{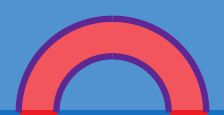}}$ (or the same with colors reversed)
; Figure \ref{Fi:GeomEssStep2} indicates all three types.  There are also two possible types of height 1 subdisks which abut two outermost subdisks, having pattern $\raisebox{-2pt}{\includegraphics[height=12pt]{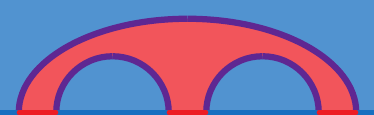}}$; Figure \ref{Fi:GeomEssStep3} indicates both types.  There are nopossible height 1 subdisks which abut more than two outermost subdisks.

\begin{figure}[h]
\begin{center}
\;\hfill
\includegraphics[width=
3.025in]{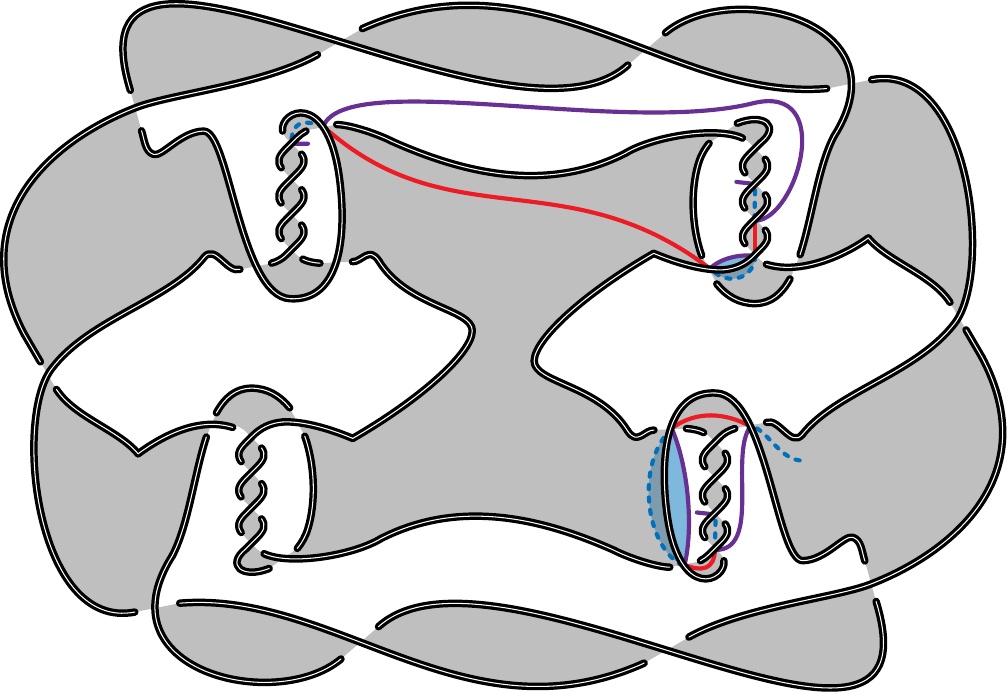}
\hfill
\includegraphics[width=
3.025in]{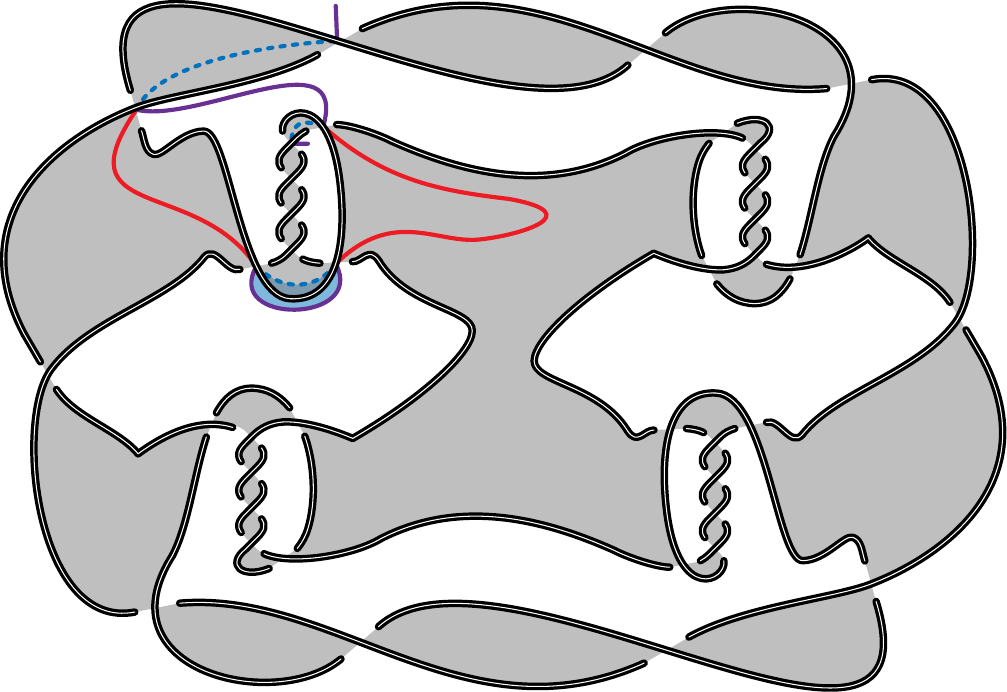}
\hfill\;
\caption{
Height 1 subdisks abutting a single outermost disk}\label{Fi:GeomEssStep2}
\end{center}
\end{figure}

\begin{figure}[h]
\begin{center}
\;\hfill
\includegraphics[width=
3.025in]{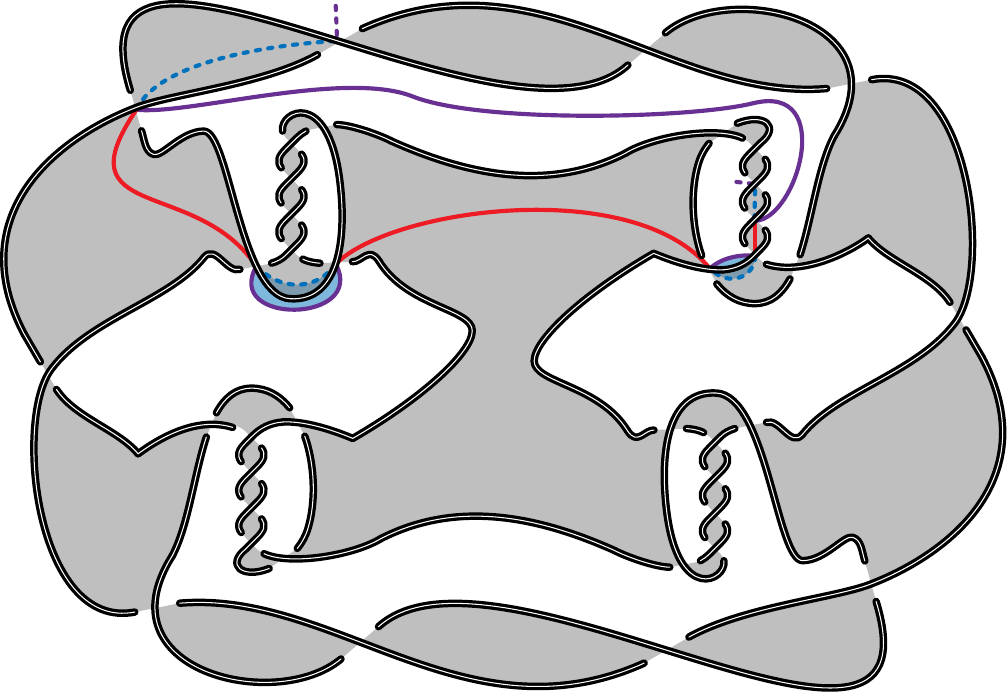}
\hfill
\includegraphics[width=
3.025in]{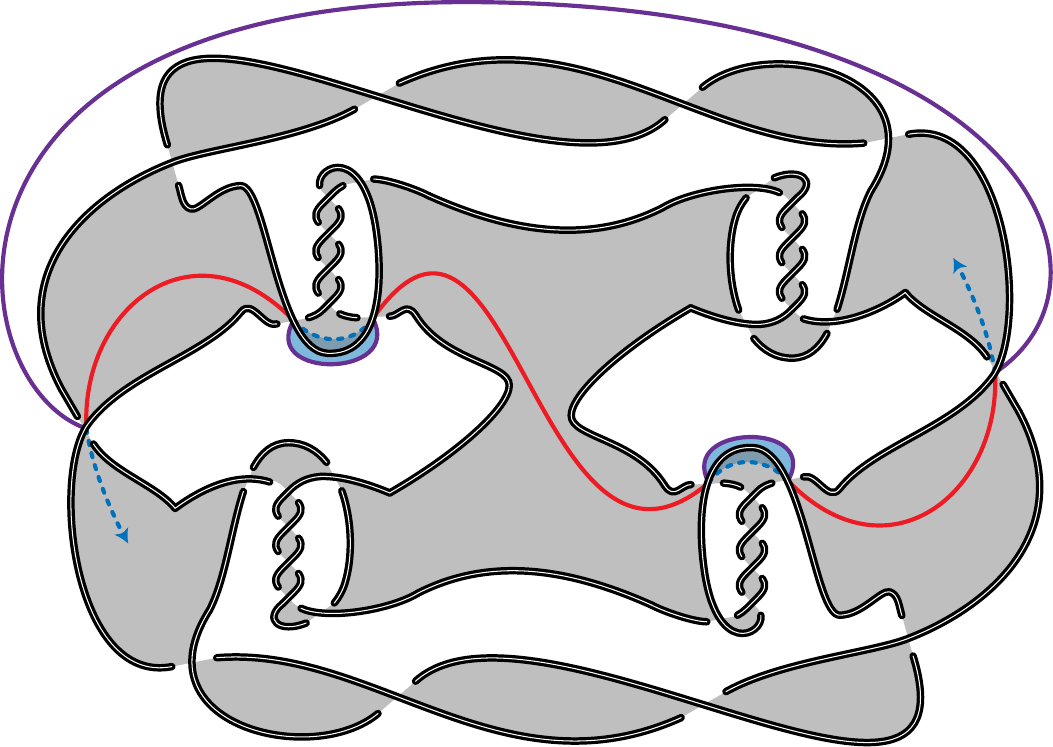}
\hfill\;
\caption{Height 1 subdisks abutting two outermost disks}
\label{Fi:GeomEssStep3}
\end{center}
\end{figure}

There are no geometric compressing disks in which every subdisk has height at most 1. (There is, however, such an {\it algebraic} compressing disk!) If there were a subdisk $D_2$ of height 2, it could not abut any type of subdisk shown in Figure \ref{Fi:GeomEssStep2}, nor on the left in Figure \ref{Fi:GeomEssStep3}, hence would necessarily abut the type shown right in Figure \ref{Fi:GeomEssStep3}. Yet, the furthest that one can then build out $D_2$ is as shown left in Figure \ref{Fi:GeomEssStep4}, where we have the pattern $\raisebox{-2pt}{\includegraphics[height=12pt]{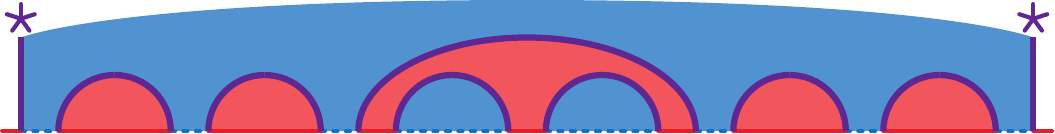}}$: from there, one can proceed no further within $D_2$ because the two starred purple arcs lie in different disks of $W\cut F$ and can cut off no subdisk of height 0 nor 1.  We thus conclude that $F$ is geometrically incompressible.

\begin{figure}[h]
\begin{center}
\;\hfill
\includegraphics[width=
3.15in]{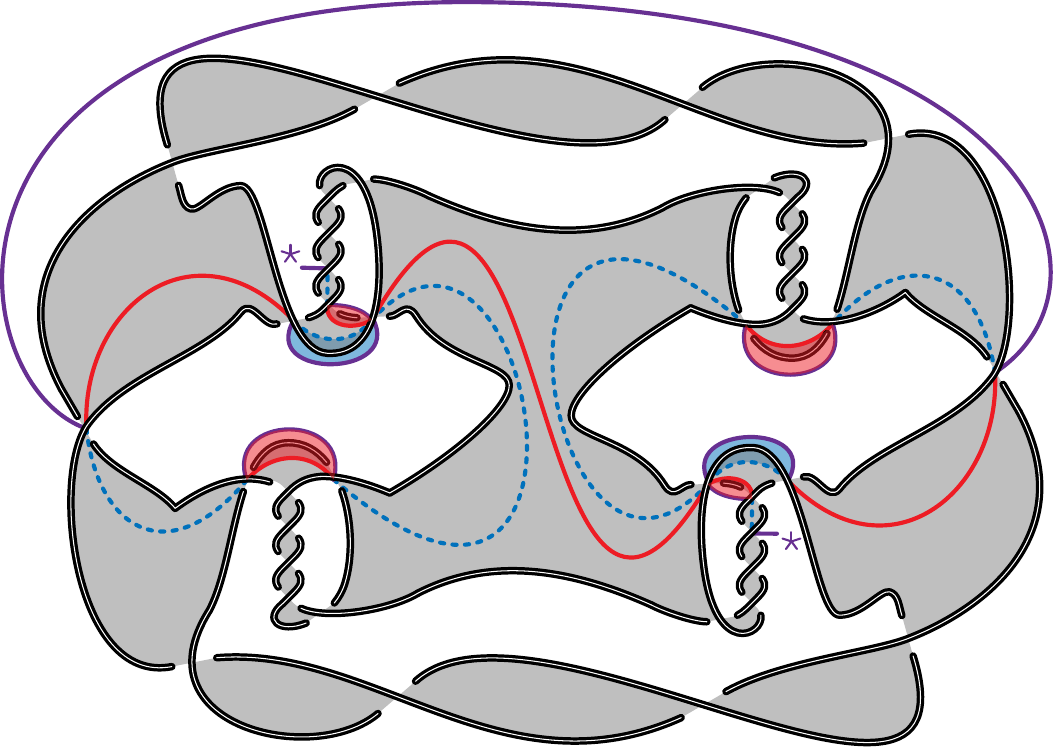}\hfill
\includegraphics[width=
3.15in]{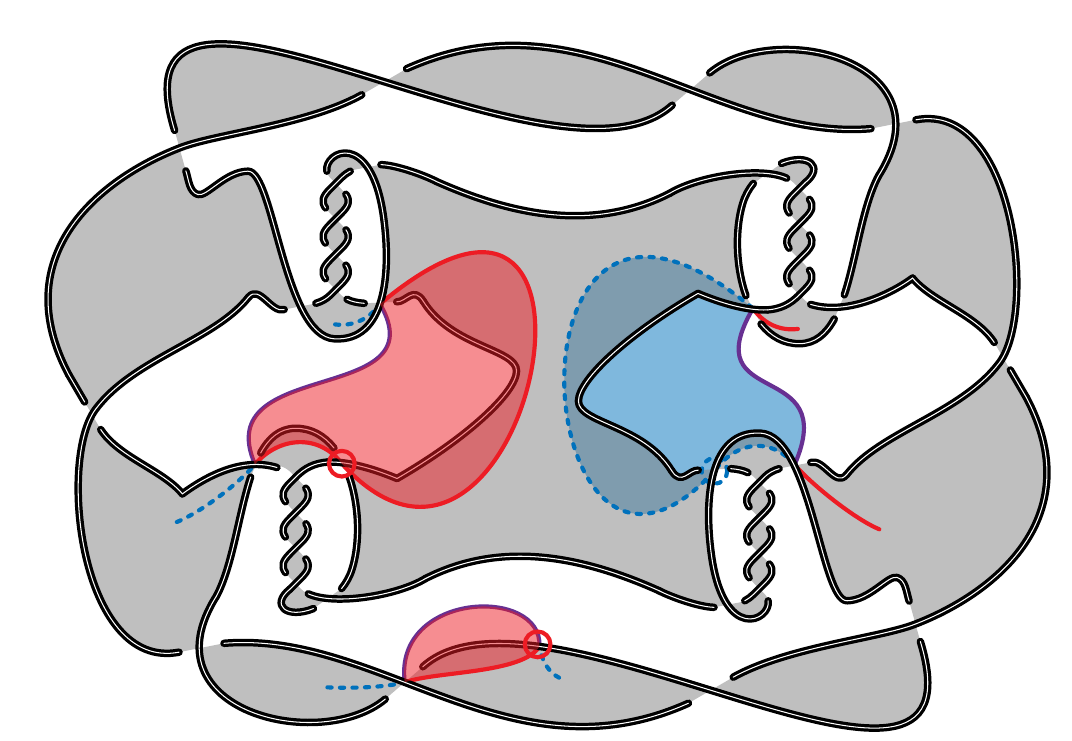}
\hfill\;
\caption{Left: the main contradiction in the proof of Proposition \ref{P:GeomEss}. Right: possible height 0 subdisks in a boundary compressing disk.}
\label{Fi:GeomEssStep4}
\end{center}
\end{figure}

Finally, we must adapt the argument above to show that $F_1$ is not geometrically $\partial$-compressible.  
Suppose otherwise.  As before, choose a $\partial$-compressing disk $D$ for $F_1$ which minimizes $|D\pitchfork{W}|$, and consider the possible configurations of the subdisks of $D\cut{W}$, starting with those of height 0. 

In addition to the types from Figure \ref{Fi:GeomEssStep1}, there is, up to symmetry, one additional type of possible outermost disk of $D\cut{W}$. Figure \ref{Fi:GeomEssStep4}, right, shows two of the four examples of this type of subdisk and, at the bottom, an example of another type of outermost subdisk which would contradict minimality. This additional type of subdisk, however, cannot abut a subdisk of height 1.  In fact, the only types of height 1 subdisks are still those in Figures \ref{Fi:GeomEssStep2}-\ref{Fi:GeomEssStep3}.  Considering subdisks of height 2 leads to the same contradiction as before, albeit with an additional case $\raisebox{-2pt}{\includegraphics[height=12pt]{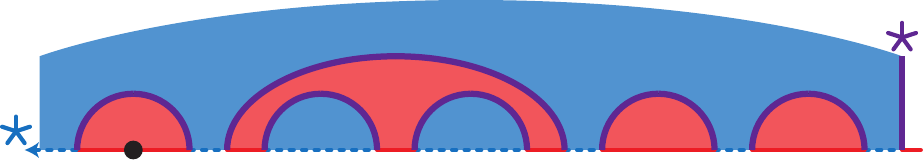}}$, shown in Figure \ref{Fi:GeomEssStep6}, that leads to the same sort of contradiction.
\end{proof}

\begin{figure}[h]
\begin{center}
\includegraphics[width=
3.6in]{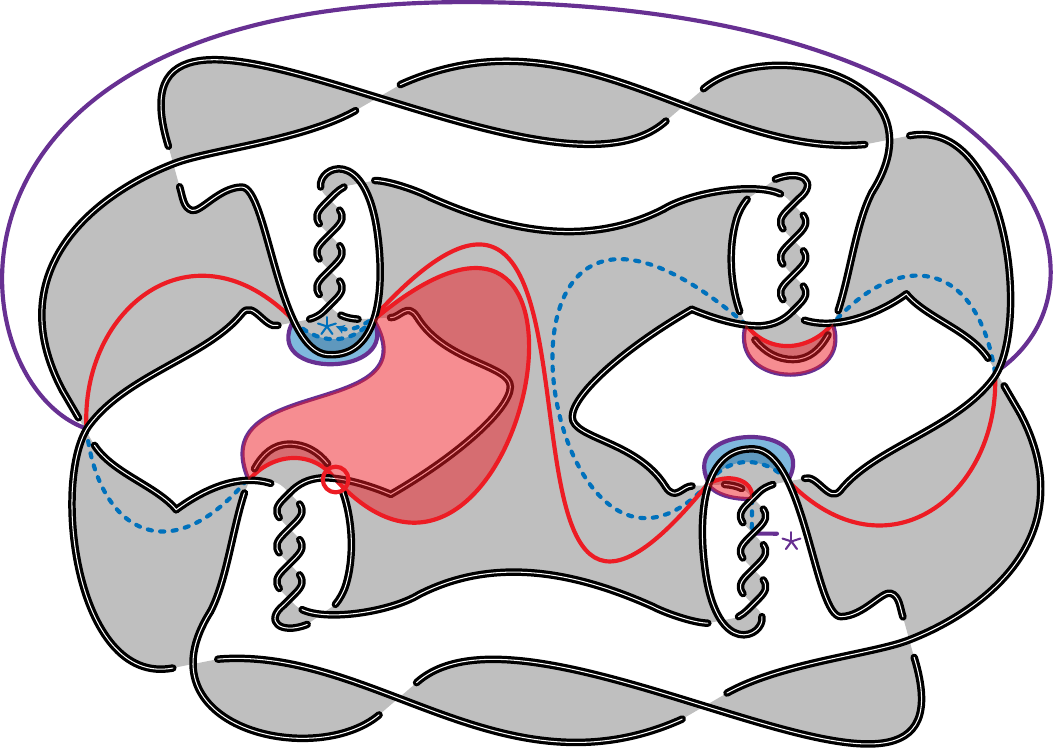}
\caption{The final contradiction in the proof of Proposition \ref{P:GeomEss}.}
\label{Fi:GeomEssStep6}
\end{center}
\end{figure}

If one replaces each $2$-twist annulus in Example \ref{Fi:BadPlumb}, left, with a ($1$-twist)  Hopf band, then the resulting surface is geometrically $\partial$-compressible.  To see this, draw the plumbed-on Hopf band in top-right of the figure in the way shown center in Figure \ref{Fi:BadPlumb}, rather than with the dark lowerered hemisphere. Then there is a $\partial$-compressing disk whose boundary runs over this plumbed-on Hopf band, through the middle of the surface, along the top of the light hemisphere bottom-right, and then across the crossing band at the far-right, where it passes over the overpass. 

What if one instead replaces each two-twist band in Figure \ref{Fi:BadPlumb}, left, with a $\frac{3}{2}$-twist band?

\begin{prop}
Let $F$ be the surface shown left in Figure\ref{Fi:BadPlumb}, but with each plumbed-on two-twist band replaced with a $\frac{3}{2}$-twist band.  Then $F$ is geometrically essential.
\end{prop}

\begin{proof}
The proof is the same as that of Proposition \ref{P:GeomEss}, with the obvious modifications in the attending figures.
\end{proof}

The method behind our proof of Proposition \ref{P:GeomEss}  is completely systematized and thus should be programmable:

\begin{problem}\label{P:Compute}
Write a computer program that, given as input a positive integer $n$ and a link diagram $P$ (described by, say, a Gauss code or a PD code), characterizes all subdisks of height at most $n$ for its checkerboard surfaces, and thus determines whether or not either surface has a compressing disk or $\partial$-compressing disk in which all subdisks have height at most $n$.  Then use the resulting data to catalog essential checkerboard surfaces by crossing number.
\end{problem}

\section{Plumbing essential bands onto compressible surfaces}\label{S:Final}

The two surfaces plumbed together in Example \ref{Ex:StrongConverse} are copies of one another.  This surface admits a compressing disk on one side only, and in fact the compressing disk is unique up to isotopy.  This is the key, in some sense, to why that example works.

\begin{example}\label{Ex:FalseConverse}
In Example \ref{Ex:GabaiConverse}, we plumbed two essential pairs of pants onto a trivial annulus.  Suppose we replace one of the two with a single 2-twist annulus, as shown in Figure \ref{Fi:OneBandIsNotEnough}.   Surprisingly, the resulting surface is compressible---the figure shows a compressing disk (if we had replaced {\it each} pair of pants in Example with a 2-twist annulus, compressibility would follow from the classification of essential spanning surfaces for 2-bridge knots and links due to Hatcher-Thurston \cite[Theorem 1(c) and Remark 1]{ht}). 
\end{example}

\begin{figure}[h]
 \begin{center}
\includegraphics[height=2in]{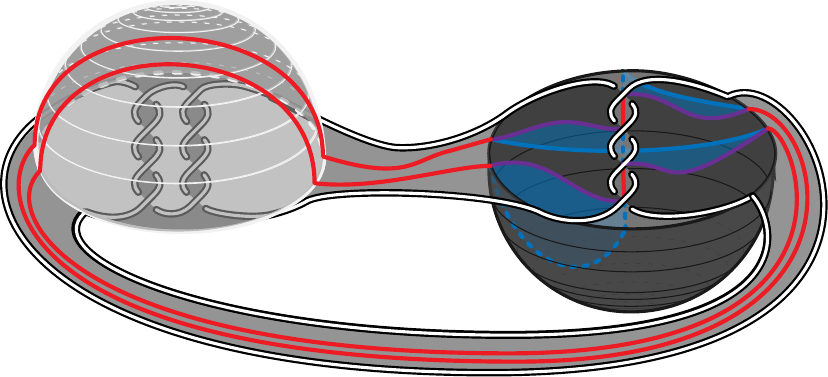}
\caption{A surface with a harder-to-find compressing disk}
\label{Fi:OneBandIsNotEnough}
\end{center}
\end{figure}

This motivates the following question, which one can ask of geometric or algebraic compressibility.

\begin{question}\label{Q:Band}
Suppose $F_1$ is compressible, $F_2$ is an essential annulus or M\"obius band whose core circle is unknotted, and $F=F_1*F_2$ is obtained by plumbing $F_2$ onto $F_1$, say along an arc $\alpha$.  Under what conditions is the resulting surface compressible?  Boundary compressible?
\end{question}

In \textsection\ref{S:Sufficient}, we provide sufficient conditions for the surface $F$ in Question \ref{Q:Band} to be compressible (see Theorem \ref{T:compressible_attach_annulus}) or to be $\partial$-compressible (see Theorem \ref{T:compressible_attach_mob}).  In \textsection\ref{S:2bridge}, we apply those results to 2-bridge knots and links.

\subsection{Sufficient conditions for compressibility and $\partial$-compressibility}\label{S:Sufficient}

\begin{theorem}\label{T:compressible_attach_annulus}
Suppose $D$ is a compressing disk for a spanning surface $F_1$, $\alpha$ is a properly embedded arc in $F_1$ with $|\alpha\pitchfork \partial D|=1$, and $F_2$ is an unknotted annulus with any number of twists. Then any surface $F$ obtained by plumbing $F_2$ onto $F_1$ along $\alpha$ is compressible---this is true for algebraic and geometric compressibility.

Moreover, if $F=F_1*F_2$ as above, $\gamma$ is a core curve of $F_2$, and $\alpha'$ is a properly embedded arc in $F_2$ that is disjoint from $F_1$ with $|\alpha'\pitchfork\gamma|=1$, then $F$ has a compressing disk $X$ with $|\alpha'\pitchfork\partial X|=1$.
\end{theorem}

\begin{proof}
We will prove this for algebraic compressibility, but precisely the same argument will work for geometric compressibility (still, see footnote \ref{N:geometric}). 
Write $\partial F_i=L_i$, and let $n$ denote the number of full twists in $F_2$ (without sign), so that $L_2$ is a $(2,\pm2n)$-torus link.  

\begin{figure}[h]
\centering
\;\hspace{.5in}
\labellist\small
\pinlabel{ {\color{white} $D\cap B_1$}} at 30 135
\pinlabel{ {\color{white} $D\cap B_+$}} at 105 105
\pinlabel{ {\color{white} $D\cap B_-$}} at 105 50
\pinlabel{ {\color{white} $u$}} at 105 10
\pinlabel{ {\color{white} $v$}} at 105 135
\endlabellist
\includegraphics[height=1.75in]{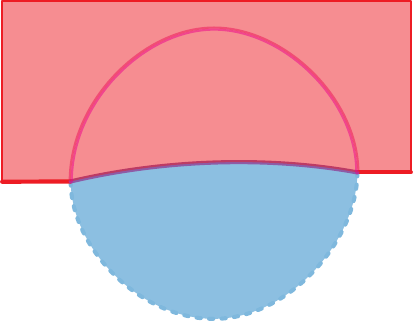}\hfill
\labellist\small
\pinlabel {$U$} at 50 10
\pinlabel {$V$} at 50 140
\endlabellist
\includegraphics[height=1.75in]{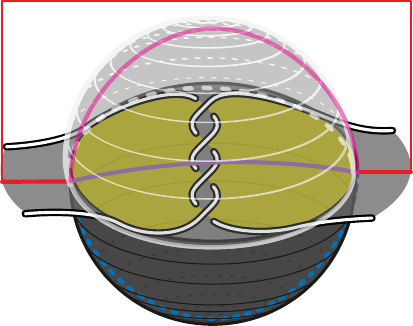}
\hspace{.5in}
\;
\caption{Left: the part of a compressing disk $D$ for $F_1*F_2$ that lies near the annulus $F_2$.  Right: A neighborhood of $F_2$ in $S^3$, with $U$ shown darkest gray, $V$ shown lightest gray, $W_1$ and $W_0$ shown yellow, and $D\cap W$ shown purple. }
\label{Fi:Disks_in_B2}
\end{figure}
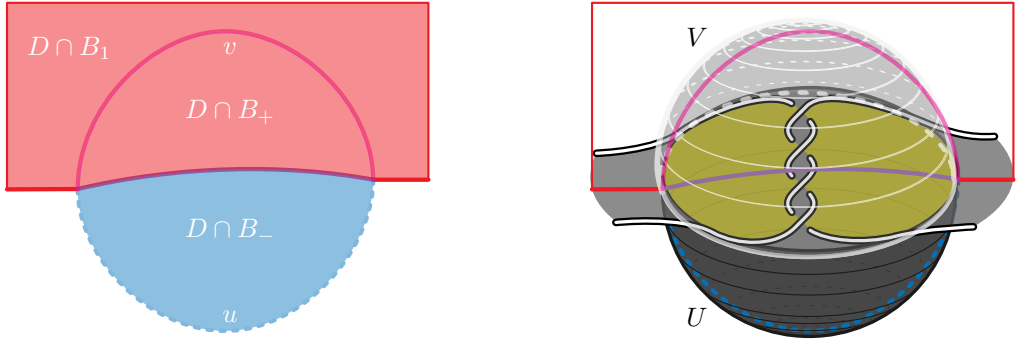

It will be convenient to visualize this in the manner suggested in Figure \ref{Fi:Disks_in_B2}, so that $F_2$ consists of the plumbing disk $U$, which lies below $S^2$, together with a band that lies entirely ``flat'' in a disk $W\subset S^2$, except in $n$ half-twisted crossing bands (all twisting the same way); a plumbing cap $V$ associated to $U$ lies above $S^2$. 
Then $F_2$ has the structure of a checkerboard surface relative to the 2-sphere $U\cup W$.  The opposite checkerboard surface is comprised of the two disks into which $F_2$ cuts $W$, together with $2n$ crossing bands that each twist the opposite direction as those in $F_2$.  
These two checkerboard surfaces intersect in vertical arcs $c=c_1\sqcup\cdots\sqcup c_{2n}$, which together cut the second checkerboard surface into two disks, which we denote $W_1$ and $W_0$. Each $W_k$ is a cap for $F_2$ and for $F$. Together with $F_2$, they cut $B_2$ into two 3-balls---write $B_-$ for the one incident to $U$ and $B_+$ for the one incident to $V$.

If $D$ is disjoint from $B_2$ or $n=0$, then $F$ is obviously compressible, so assume otherwise. 
We may assume that we have chosen $D$ so that $D\cap W$ is comprised entirely of arcs, each of whose endpoints lie on opposite arcs of $\partial W\cap\mathring{F}_1$.

Note that $D\cap \mathring{B}_2$ is a disk $D_0$ whose boundary is comprised of an arc $u\subset U$ and an arc $v\subset V$. 
Take $n$ disjoint parallel copies of $D$,\footnote{\label{N:geometric} It may be convenient to think of these copies as all lying in a thin regular neighborhood $\nu D$ (letting $\wt{D}$ denote the (unique) properly embedded disk in $S^3\cut F$ satisfying $ h_F(\wt{D})=D$, $\nu D$ is the image under $h_F$ of a regular neighborhood in $S^3\cut F$ of $\wt{D}$).  In particular, if $D$ is a geometric compressing disk, then one can assume that $\partial \nu D\cap F$ is an embedded annulus, which will guarantee in that case that the compressing disk $X$ that we construct is a geometric one.} denote them $D^1,\dots,D^n$, and write $D^*=\bigsqcup_{j=1}^nD^j$. Then each $D^j\cap B_2$ is a disk $D^j_0$, where each $\partial D^j_0$ is comprised of an arc $u^j\subset U$ and an arc $v^j\subset V$. Index so that $\alpha$ intersects the arcs $u^1,\dots,u^n$ in that order. 
Write $D^*_0=\bigsqcup_{j=1}^nD^j_0=D^*\cap B_2$.

\begin{figure}[h]
\centering
\;\hfill
\labellist\small
\pinlabel {$W_1$} at 35 40 
\pinlabel {$W_0$} at 95 40 
\pinlabel {$\partial_1u^1$} at -2 13 
\pinlabel {$\partial_1u^2$} at 15 32 
\pinlabel {$\partial_0u^1$} at 126 26 
\pinlabel {$\partial_0u^2$} at 128 36 
\pinlabel {$c_1$} at 76 13 
\pinlabel {$c_2$} at 76 30 
\pinlabel {$c_3$} at 76 47 
\pinlabel {$c_4$} at 76 64 
\endlabellist
\includegraphics[width=
3.025in]{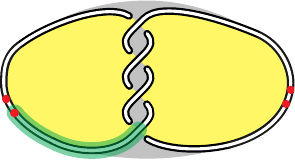}\hfill
\labellist\small
\pinlabel {$w_1^1$} at 35 30 
\pinlabel {$w_1^2$} at 35 45 
\pinlabel {$w_0^1$} at 102 20 
\pinlabel {$w_0^2$} at 102 47 
\pinlabel { {\color{myBlue} part of $\beta_0$}} at 70 -5
\pinlabel { {\color{red} $\beta_1$}} at 79 21 
\pinlabel { {\color{myBlue} $\beta_2$ }} at 81 39 
\pinlabel { {\color{red} $\beta_3$}} at 79 53 
\pinlabel { {\color{myBlue} part of $\beta_4$}} at 70 80 
\endlabellist
\includegraphics[width=
3.025in]{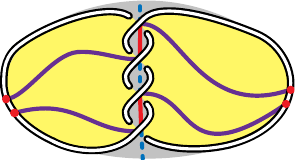}\hfill\;
\caption{Labeling arcs in the proof of Theorem \ref{T:compressible_attach_annulus}}
\label{Fi:Disks_in_B2_2}
\end{figure}
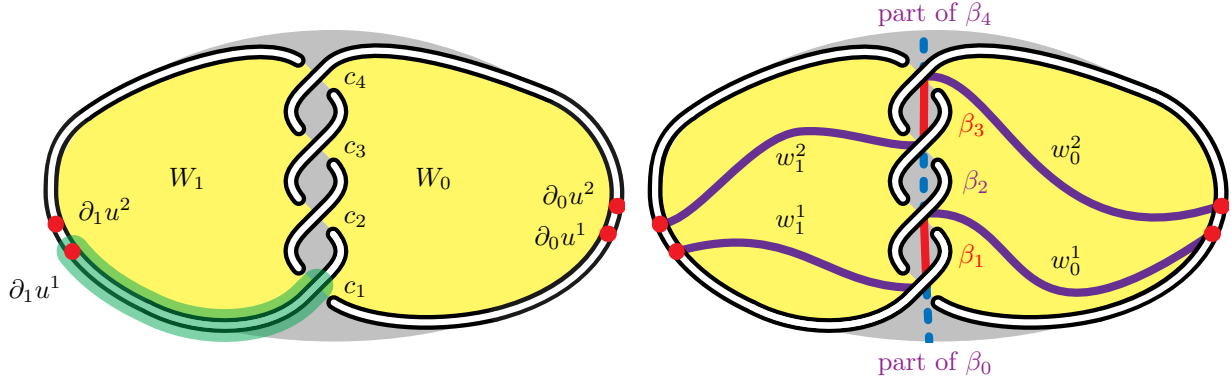

Now, in the following way, replace $D^*_0$ with a single disk $X^*\subset B_2$.  See Figures \ref{Fi:Disks_in_B2_2} and \ref{Fi:Disks_in_B2_3}:
\begin{enumerate}[leftmargin=*, label=(\arabic*)]
\item For $k=1,0$ and $j=1,\dots,n$, write $\partial_k u^j$ for the endpoint of $u^j$ that lies on $\partial W_k$.  Assume without loss of generality that $W_1,W_0$, $u^1,\dots,u^n$, and $c_1,\dots,c_{2n}$ are indexed as shown left in Figure \ref{Fi:Disks_in_B2_2}, so that, if one cuts $\partial F_2$ at the endpoints of all the arcs $u^j$ and $c_\ell$, one of the resulting arcs (highlighted green in the figure) has one endpoint at the overpass endpoint of $c_1$ and the other at $\partial_1u^1$. (If this were not the case, we could make it so by reversing the indices on $W_1$ and $W_0$.)
\item For $k=1,0$, construct systems of disjoint, properly embedded arcs $w_k=w_k^1\sqcup\cdots\sqcup w_k^n\subset W_k$ such that 
each $w_k^j$ shares one endpoint with $u^j$ and
has one endpoint in the interior of the vertical arc $c_{2j-k}$.
\item  Construct a system of disjoint, properly embedded arcs $\wh{u}=\wh{u}^2\sqcup\cdots\sqcup\wh{u}^n\subset U$ such that, for $j=2,\dots,n$, we have $\partial\wh{u}^j=\partial_1u^j\cup\partial_0u^{j-1}$. See Figure \ref{Fi:Disks_in_B2_3}, right.
\item\label{step:4} Construct an arc $\beta\subset F_2\setminus \mathring{U}$ with $\partial\beta=\partial_1u^1\cup\partial_0u^n$ that intersects each vertical arc $c_\ell$ in the single point which is also an endpoint of an arc $w_k^j$ (namely $k=\ell\mod 2$ and $j=\lfloor\frac{\ell+1}{2}\rfloor$). Denote the components of $\beta\cut c$ by $\beta_0,\beta_1,\dots,\beta_{2n}$, such that each vertical arc $c_\ell$ is incident to $\beta_\ell$ and $\beta_{\ell-1}$. 
\item For $j=1,\dots,n$, the arcs $w_1^j$, $\beta_{2j-1}$, $w_0^{j}$, and $v^j$ co-bound a disk $X_{2j-1}\subset B_+$.
\item For $j=2,\dots,n-1$, the arcs $w_1^j$, $\beta_{2j-2}$, $w_0^{j-1}$, and $\wh{u}^j$ co-bound a disk $X_{2j-2}\subset B_-$. Likewise, the arcs $w_1^1$ and $\beta_0$ co-bound a disk $X_0\subset B_-$, and the arcs $w_0^{n}$ and $\beta_{2n}$ co-bound a disk $X_{2n}\subset B_-$.  
\item The union $X^*=\bigcup_{\ell=0}^{2n}X_\ell$ is connected because, for each $\ell=1,\dots,2n$, we have $X_\ell\cap X_{\ell-1}=w_k^j$ for $k=\ell\mod 2$ and $j=\lfloor{\frac{\ell+1}{2}}\rfloor$. Also, 
$\chi(X^*)=(4n)-((2n+1)+4n-1)+(2n+1)=1$, so $X^*$ is a disk.
\end{enumerate}

\begin{figure}[h]
\centering
\;\hfill
\labellist\small
\pinlabel{ {\color{white} $v_1$}} at 70 90
\pinlabel{ {\color{red} $v_2$}} at 80 106
\pinlabel {{ \color{red} $\partial_1u^2$}} at -11 28 
\pinlabel {{ \color{red} $\partial_1u^1$}} at -8 16 
\pinlabel {{\color{red} $\partial_0u^1$}} at 150 23 
\pinlabel {{\color{red} $\partial_0u^2$}} at 152 36 
\endlabellist
\includegraphics[width=.4\textwidth]{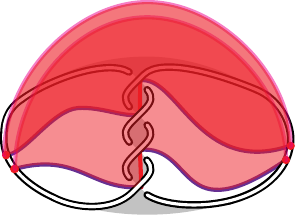}\hfill
\labellist\small
\pinlabel{ {\color{white} $\beta_0$}} at 57 33
\pinlabel{ {\color{white} $\wh{u}^2$}} at 75 8
\pinlabel{ {\color{white} $w_1^1$}} at 40 67
\pinlabel{ {\color{white} $w_2^1$}} at 50 80
\pinlabel{ {\color{myBlue} $\beta_2$}} at 80 80
\pinlabel{ {\color{myPurple} $w_1^0$}} at 90 75
\pinlabel{ {\color{myPurple} $w_2^0$}} at 100 60
\pinlabel{ {\color{myBlue} $\beta_4$}} at 110 91
\pinlabel {{ \color{red} $\partial_1u^2$}} at -11 68 
\pinlabel {{ \color{red} $\partial_1u^1$}} at -8 56 
\pinlabel {{\color{red} $\partial_0u^1$}} at 150 63 
\pinlabel {{\color{red} $\partial_0u^2$}} at 152 76 
\endlabellist
\includegraphics[width=.4\textwidth]{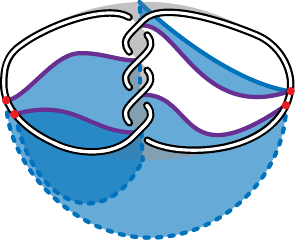}\hfill \;
\caption{The disk $X^*\subset B^2$ intersects $B_+$ in $n=2$ disks (left), and $X^*$ intersects $B_-$ in $n+1=3$ disks (right).}
\label{Fi:Disks_in_B2_3}
\end{figure}
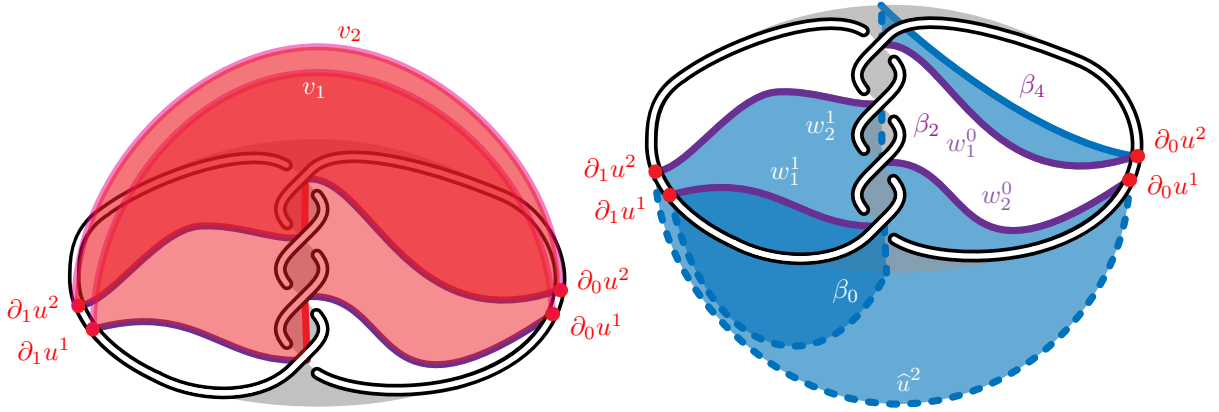

Each of the disks $D_j\setminus\mathring{B}_2$ attaches along a single arc $v^j$ to this disk $X^*$.  Therefore, the union $X=X^*\cup\bigcup_{j=1}^n(D^j\setminus\mathring{B}_2)$ is also a disk, and $\partial X\subset\mathring{F}$. Finally, we note that $\partial X$ intersects each vertical arc $c_\ell$ transversally in a single point, so $\partial X$ is homotopically nontrivial in $F$, and $X$ is a compressing disk for $F$.  Moreover, by construction, $X$ has the extra property described in the second part of the theorem.
\end{proof}

\begin{rem}
Theorem \ref{T:compressible_attach_annulus} becomes false without the hypothesis that $|\alpha\pitchfork \partial D|=1$.  Indeed, the essential surface $F$ constructed in Example \ref{Ex:GabaiConverse} could instead be built from the compressible surface $F_0$ shown left in Figure \ref{Fi:GabaiConverse} by separately plumbing on two annuli, $A,A'$, each with two full positive twists.  The surface $F_0*A$ would be compressible by Theorem \ref{T:compressible_attach_annulus}, and yet plumbing the essential annulus $A'$ onto this surface would yield $F$, which is essential.
\end{rem}

It is also natural to ask how the conclusion of Theorem \ref{T:compressible_attach_annulus} changes when $F_2$ is a M\"obius band, rather than an annulus.  

\begin{theorem}\label{T:compressible_attach_mob}
Suppose $D$ is a compressing disk for a spanning surface $F_1$, $\alpha$ is a properly embedded arc in $F_1$ with $|\alpha\pitchfork \partial D|=1$, and $F_2$ is an unknotted M\"obius band with any number of twists. Then any surface $F$ obtained by plumbing $F_2$ onto $F_1$ along $\alpha$ is $\partial$-compressible---this is true for algebraic and geometric compressibility.
\end{theorem}

\begin{proof}
Let $2n+1$ denote the number of half-twists in $F_2$ (without sign). If $n=0$, then $F$ is obviously $\partial$-compressible, so assume $n\geq 1$. The proof is now the same as that of Theorem \ref{T:compressible_attach_annulus}, with the following minor change. 

There are $2n+1$ crossings $c_1,\dots,c_{2n+1}$, rather than $2n$ crossings.
On step \ref{step:4}, $\beta$ should intersect the interiors of $c_1,\dots,c_{2n}$ as before, but $\beta$ should run across the overpass at $c_{2n+1}$, rather than intersecting its interior. This is why the resulting disk $X$ is a $\partial$-compressing disk, rather than a compressing disk. 
\end{proof}

Proposition \ref{P:GeomEss} may feel more surprising in light of Theorems \ref{T:compressible_attach_annulus} and \ref{T:compressible_attach_mob} than it did initially.  How can it be that we were able to take the highly inessential surface $F_0$ shown right in Figure \ref{Fi:Ex1}, plumb on six unknotted annuli (including two via boundary sum) as shown in Figure \ref{Fi:Ex2}, and get a surface that is geometrically essential?  Part of the answer is that $F_0$, while algebraically compressible and geometrically $\partial$-compressible, is geometrically incompressible. The other part of the answer is the following.

\begin{rem}
 It is vital in the hypothesis of Theorems \ref{T:compressible_attach_annulus} and \ref{T:compressible_attach_mob} that $D$ is a compressing disk, rather than a $\partial$-compressing disk.  
This is because, unless $D$ is disjoint from $B_2$, 
the construction in the proofs of those theorems requires taking $n$ copies of $D$ (and transforming them) which will mean that $\partial X$ intersects $L$ in $n$ copies of the point $\partial D\cap L_1$, plus an additional point at the overpass of $c_{2n+1}$ if $F_2$ is a M\"obius band.  Thus, the proof falls apart if $n\geq 2$, or if $n=1$ and $F_2$ is a M\"obius band.  Further, if $n=0$, then $F$ is obviously compressible or $\partial$-compressible simply on account of $F_2$, so there is no result worth stating.  We do, however, recover the following.
\end{rem}

\begin{prop}\label{P:compressible_attach_mob}
Suppose $D$ is a $\partial$-compressing disk for a spanning surface $F_1$, $\alpha$ is a properly embedded arc in $F_1$ with $|\alpha\pitchfork \partial D|=1$, and $F_2$ is a Hopf band. Then any surface $F$ obtained by plumbing $F_2$ onto $F_1$ along $\alpha$ is $\partial$-compressible.
\end{prop}

\begin{proof}
The proof is the same as that of Theorem \ref{T:compressible_attach_annulus}.
\end{proof}

\subsection{Applications to 2-bridge knots and links}\label{S:2bridge}
As a corollary of Theorem \ref{T:compressible_attach_annulus}, we obtain a new, constructive proof of \cite[Theorem 1(c)]{ht}, in the case of spanning surfaces. For context, we provide a brief summary of this important result of Hatcher and Thurston, as it pertains to spanning surfaces.  In Theorem 1(b) of \cite{ht}, they prove that every geometrically incompressible spanning surface for a 2-bridge knot or link $L_{p/q}$ can be constructed as follows.  First, take a subtractive continued fraction expansion of $p/q$,
\[\frac{p}{q}=a+[b_1,b_2,\dots,b_k]=a+\cfrac{1}{b_1-\cfrac{1}{b_2-\cdots-\cfrac{1}{b_k}}}.\]
Then take $k$ unknotted bands $F_1,\dots,F_k$, where each $F_i$ has $\frac{i}{2}$ half-twists about its core circle, and plumb $F_i$ onto $F_{i-1}$ for each $i=2,\dots,k$ in such a way that (no plumbing is a boundary sum and) the plumbing squares $F_i\cap F_{i-1}$ and $F_i\cap F_{i+1}$ are disjoint for each $i=2,\dots,k-1$. 
Theorem 1(c) characterizes precisely when the resulting surface is geometrically essential:

\begin{cor}[Theorem 1(c) of \cite{ht}]\label{C:compressible_attach_annulus}
A surface $F=F_1*\cdots*F_k$ constructed from a continued fraction expansion $\frac{p}{q}=a+[b_1,\dots,b_k]$ is geometrically essential if and only if each $|b_i|\geq 2$.
\end{cor}

\begin{proof}
If each $b_i\geq 2$, then each $F_i$ is isotopic to a checkerboard surface from a reduced alternating link diagram and thus is essential, and so Ozawa's Theorem \ref{T:Ozawa} implies that $F$, too, is essential.

If some $b_i=\pm 1$, then there is a $\partial$-compressing disk for $F$ whose boundary lies entirely in, and is isotopic to the core curve of, $F_i$. 

Finally, suppose that some $b_i=0$, so that $F_i$ is a compressible annulus. If $i=1$ or $i=k$, or if $F_{i-1}$ and $F_{i+1}$ are plumbed onto the same side of $F_i$, then there is a compressing disk for $F$ whose boundary is the core curve of $F_i$. Assume instead that $i\neq1,k$ and that $F_{i-1}$ and $F_{i+1}$ are plumbed onto opposite sides of $F_i$ (as in Figure \ref{Fi:OneBandIsNotEnough} e.g.).  Choose a compressing disk $D$ for $F_i$ that runs along the side of $F_i$ where we will plumb on $F_{i+1}$.  Then $D$ is a compressing disk for $F_1*\cdots*F_i$. Write $\alpha$ for the arc in $F_i$ along which we will plumb $F_{i+1}$, and note that $|\alpha\pitchfork \partial D|=1$.  Then, using (the construction described in the proof of) Theorem \ref{T:compressible_attach_annulus}, replace $D$ with a compressing disk $X$ for $F_1*\cdots*F_i*F_{i+1}$.  Repeat inductively (using the second part of the theorem), either until the next plumbing factor is attached in a way that does not intersect the existing plumbing disk---in which case we keep the existing compressing disk while completing the construction of $F$---or until we have constructed all of $F$. 
\end{proof}

We also note that a spanning surface for a 2-bridge knots or link is geometrically essential if and only if it is $\pi_1$-essential. Hatcher and Thurston assert this without explicit proof in \cite{ht}. This seems like an appropriate place to provide a short proof.

\begin{cor}[Theorem 1(c) of \cite{ht}]
Suppose $F=F_{\mathbf{b}}^\mathbf{v}$ is a Hatcher-Thurston spanning surface for a 2-bridge knot or link $K$.  The following are equivalent.
\begin{enumerate}[label=(\arabic*)]
    \item\label{part:pi_1} $F$ is $\pi_1$-injective.
    \item\label{part:geom_ess} $F$ is geometrically essential.
    \item\label{part:b_ess} $\mathbf{b}$ is essential.
\end{enumerate}
\end{cor}

\begin{proof}
The equivalence of \ref{part:geom_ess} and \ref{part:b_ess} is precisely Corollary \ref{C:compressible_attach_annulus}, or Remark 1 and Theorem 1 (c) in \cite{ht}.
Further, since $K$ is 2-bridge, it is not the unknot, and $F$ is not a trivial M\"obius band. This confirms that \ref{part:pi_1} implies \ref{part:geom_ess}. It thus remains only to prove that \ref{part:b_ess} implies \ref{part:pi_1}.

To this end, suppose that $\mathbf{b}=r+[b_1,\dots,b_k]$ is essential, and write $F=F_1*\cdots*F_k$.  Then each $F_i$ is $\pi_1$-essential, because it can be seen as a checkerboard surface for a reduced alternating diagram \cite[Lemma 3.3]{ozawa11}.  Since any Murasugi sum of $\pi_1$-essential surfaces is $\pi_1$-essential \cite[Lemma 3.4]{ozawa11}, it follows that $F$ is $\pi_1$-essential.
\end{proof}

We note that the compressing disk constructed in the proof of Corollary \ref{C:compressible_attach_annulus} could become extremely complicated. 

\begin{example}\label{Ex:432}
Suppose one constructs a Hatcher-Thurston surface $F$ from the continued fraction $[8,6,8,6,8,0,6,8,6,8,6]$ by plumbing the bands incident to the trivial annulus $F_6$ onto opposite sides of $F_6$, but otherwise plumbing each pair $F_{i-1}$ and $F_{i+1}$ to the same side of $F_i$.   Then, in the construction above, beginning with a compressing disk $D$ for $F_6$ (and thus for $F_1*\cdots*F_6$, we would take three copies of $D$ while constructing a compressing disk $X$ for $F_1*\cdots*F_7$, then we would take four copies of $X$ while constructing a compressing disk for $F_1*\cdots*F_8$, and so on.  Writing $\alpha$ for an arc in $F_6$ dual to the core curve of $F_6$ and disjoint from $F_5$ and $F_7$, and writing $Z$ for the compressing disk we ultimately construct for $F$, we would have $\partial Z\cap\alpha=  3\cdot 4\cdot 3\cdot 4\cdot 3=432$. In this example, we could have similarly constructed a compressing disk $Z'$, attached to the opposite side of $F$ and running along $F_1*\cdots*F_6$ rather than along $F_6*\cdots*F_{11}$, but with $|\partial Z'\cap\alpha|=4\cdot 3\cdot 4\cdot 3\cdot 4=576$.
\end{example}

We suspect that Theorem \ref{T:compressible_attach_annulus} can be improved to establish uniqueness of the compressing disks in the preceding example and more generally:

\begin{conjecture}\label{Conj:Band}
Let $F=F_1*F_2$, $\alpha$, and $\alpha'$ be as in Theorem \ref{T:compressible_attach_annulus}, and assume that $F_2$ has at least one full twist. If every compressing disk $D$ for $F_1$ satisfies either $\alpha\cap\partial D=\varnothing$ or $|\alpha\pitchfork\partial D|=1$, then:
\begin{enumerate}[label=(\arabic*)]
\item There is a bijection between the compressing disks for $F_1$ (up to isotopy) and those for $F$: given a compressing disk $F_1$, modify it (if necessary) as described in the proof of Theorem \ref{T:compressible_attach_annulus} to obtain a compressing disk for $F$.
\item Every compressing disk $D'$ for $F$ satisfies either $\alpha'\cap\partial D'=\varnothing$ or $|\alpha'\pitchfork\partial D'|=1$.
\end{enumerate}
All of this is true for both algebraic and geometric compressing disks.
\end{conjecture}

\begin{rem}
If Conjecture \ref{Conj:Band} is true, then repeated application of this fact proves that the two compressing disks described in Example \ref{Ex:432} are the only compressing disks for that surface $F$, up to isotopy.
\end{rem}

\section*{Appendix: geometric $\partial$-compression of spanning surfaces}

 Definition \ref{D:GeomEss} is not entirely traditional. In this Appendix, we offer further context and motivation for our definition. 
 
 First, we provide equivalent, and in some way simpler, characterizations of geometric essentiality for spanning surfaces, and we prove this equivalence.
 Next, we explain how any geometrically inessential surface can be transformed by a surgery move to a simpler spanning surface for the same link.
 Then, we compare our definition to more traditional geometric notions of incompressibility and boundary-incompressibility for properly embedded surfaces in 3-manifolds.  
 Finally, we mention several related ``Neuwirth'' conjectures, most of them still open, about the existence of properly embedded surfaces in link exteriors.

The following definition first appeared in \cite{ak}:

    \begin{definition}\label{D:geom_b_c}
Let $F$ be a spanning surface for a link $L\subset S^3$, and view $F$ as being properly embedded in the link exterior $E=S^3\setminus\mathring\nu L$, so that $\partial F\subset\partial\nu L$.
We call $F$ {\bf meridianally} {\bf  boundary compressible} if:
\begin{enumerate}[label=(\arabic*)]
\item there is an embedded disk $D\subset E$ with $\mathring{D}\cap F=\varnothing$,
\item $\partial D$ consists of two arcs, $\alpha\subset F$ and $\beta\subset\partial\nu L$, and
\item for one of the arcs $\kappa$ obtained by cutting $\partial F\subset\partial\nu L$ at the shared endpoints of $\alpha$ and $\beta$, the circle $\kappa\cup\beta$ bounds a meridian disk in $\nu L$.
\end{enumerate}
\end{definition}

\begin{prop}
Let $F$ be a geometrically incompressible spanning surface for a link $L\subset S^3$.  The following are equivalent:
\begin{enumerate}[label=(\arabic*)]
\item\label{condition:1} $F$ is meridianally boundary compressible;
\item\label{condition:2} $F$ is freely isotopic to $F'\natural\MobPos$ or $F'\natural\MobNeg$ for some surface $F'$;
\item\label{condition:3} $F$ is not geometrically essential.
\end{enumerate}
\end{prop}

\begin{proof}
We will show that \ref{condition:1} $\to$ \ref{condition:2} $\to$ \ref{condition:3} $\to$ \ref{condition:1}.     

\begin{figure}[h]
    \centering
    \labellist \small
    \pinlabel {$\approx$} at 330 150
    \pinlabel {$\approx$} at 690 150
    \pinlabel {$\approx$} at 1050 150
    \pinlabel {$\to$} at 1410 150
    \endlabellist
    \includegraphics[width=\textwidth]{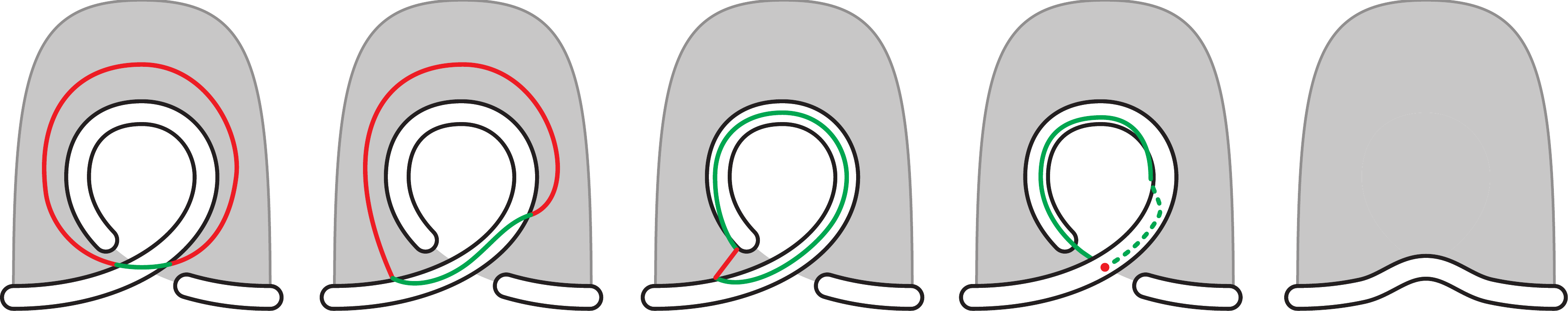}
    \caption{\label{Fi:Boundary_Compression_DeSum}meridianal boundary compression of a spanning surface is equivalent to desumming a trivial M\"obius band.}
\end{figure}
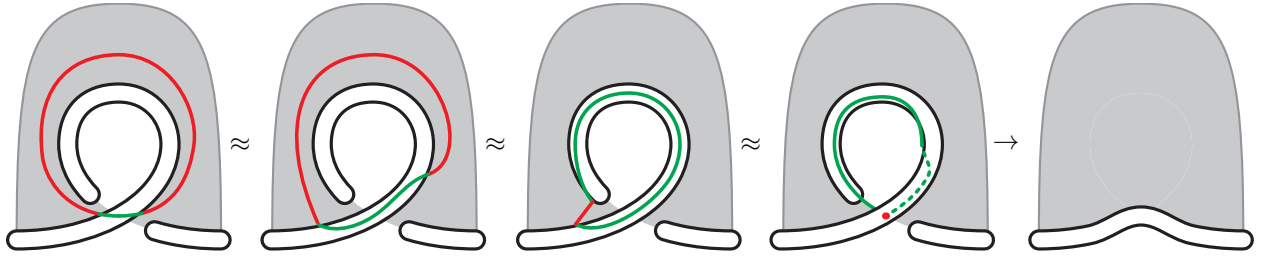

First, suppose that $F$ has a meridianal compression disk $D$, with $D$, $\alpha$, $\beta$, and $\kappa$ as described in Definition \ref{D:geom_b_c}, and let $D'$ be a meridian disk that $\beta\cup\kappa$ bounds in $\nu L$.  Then $D\cup D'$ is an embedded disk in $S^3$. Take a regular neighborhood of $B$ of this disk in $S^3$.  Then $F\cap B$ is a trivial M\"obius band, \MobPos or \MobNeg, and boundary compressing $F$ along $D$ has the same effect as de-summing this band from $F$.  See Figure \ref{Fi:Boundary_Compression_DeSum}. 

Second, suppose that $F\approx F'\natural\MobPos$ or $F\approx F'\natural\MobNeg$ for some surface $F'$.  Then one can find $D$, $\alpha$, $\beta$, and $\kappa$ as shown in Figure \ref{Fi:Unkink}.

    \begin{figure}[h]
    \centering
    \labellist\small
    \pinlabel {$\overset{\text{add negative kink}}{\longleftarrow}$} at 277 75
    \pinlabel {$\underset{\partial\text{-compress}}{\longrightarrow}$} at 277 35
    \pinlabel {$\overset{\text{add positive kink}}{\longrightarrow}$} at 595 75
    \pinlabel {$\underset{\partial\text{-compress}}{\longleftarrow}$} at 595 35
    \endlabellist
    \includegraphics[width=\textwidth]{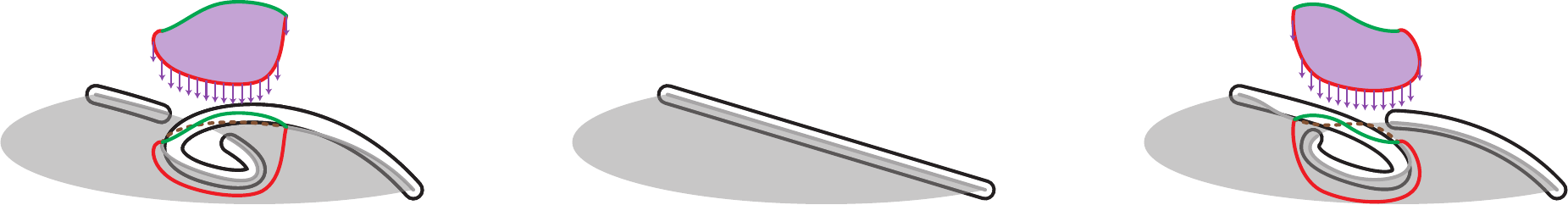}
    \caption{Arcs $\alpha$ (red), $\beta$ (blue), and $\kappa$ (green) defining a geometric $\partial$-compressing disk (purple) for a spanning surface $F$ (gray with $\partial F\setminus\kappa$ brown)}
    \label{Fi:Unkink}
    \end{figure}
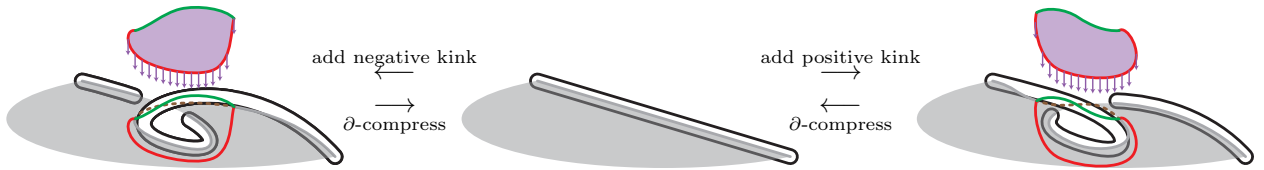

Third, suppose that $F$ is geometrically inessential. Then, since $F$ is incompressible, there is a disk $D$ with $\partial D=\alpha\cup\beta$ as described in Definition \ref{D:GeomEss}.  Delete a thin regular neighborhood of $L$ and write $F^*=F\setminus\mathring{\nu}L$, $D^*=D\setminus\mathring{\nu}L$, $\beta^*=\partial D^*\cap\partial\nu L$, and $\alpha^*=\partial D^*\setminus\mathring\beta^*$.  By assumption, $F^*$ is incompressible, and $\alpha^*$ is not $\partial$-parallel in $F^*$; therefore, $\beta$ is not parallel to $\partial F^*$ through $\partial \nu L$.  Since the natural projection $\nu L\to L$ is injective on $\beta$ (rather than wrapping it more than once around in the longitudinal direction), it follows that $D^*$ is a meridianal boundary compressing disk for $F^*$, and thus that $F^*$ and $F$ (being the same surface seen from two perspectives per Remark \ref{R:Exterior}) are meridianally boundary compressible.
\end{proof}

Condition \ref{condition:2} gives a simple and intuitive description of geometric essentiality.  Namely, 
a spanning surface is geometrically essential if and only if it arises from another spanning surface via either a tubing move (the reverse of a compression move) or a kinking move (see Figure \ref{Fi:Unkink}).  

\begin{figure}[h]
    \centering
    \includegraphics[width=\textwidth]{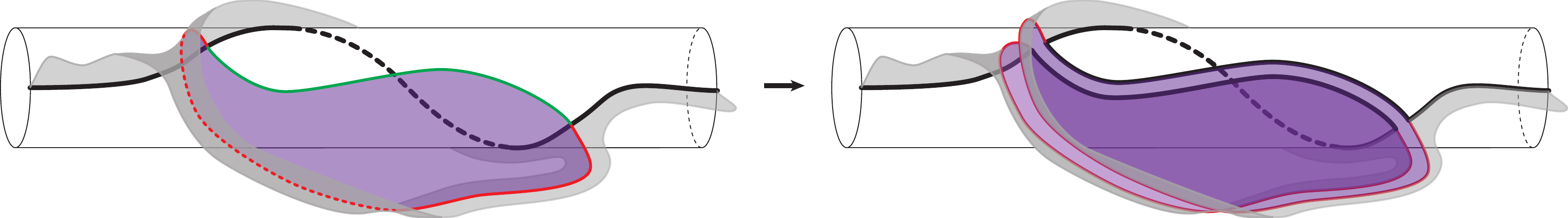}
    \caption{meridianal boundary compression of a spanning surface changes boundary slope by $\pm2$ (here by $+2$). (This image is adapted from \cite{ak}.)}
    \label{Fi:Whale}
\end{figure}

Meanwhile, condition \ref{condition:1} enables us to provide a surgery description of a geometric boundary-compression move on a spanning surface.  Namely, viewing everything in the link exterior, take a meridianal compressing disk $D$ for $F$ with $\partial D\cap F=\alpha$, cut $F$ along $\alpha$, and glue in two copies of $D$.  See Figure \ref{Fi:Whale}.  Although the entire meridianal boundary compression move $F\to F'$ is perhaps sometimes difficult to visualize, the change of $\partial F\to \partial F'$ in $\partial\nu L$ is more straightforward---see Figure \ref{Fi:Not_Bad_Boundary_Compression}.

\begin{figure}[h]
    \centering
    \includegraphics[width=\textwidth]{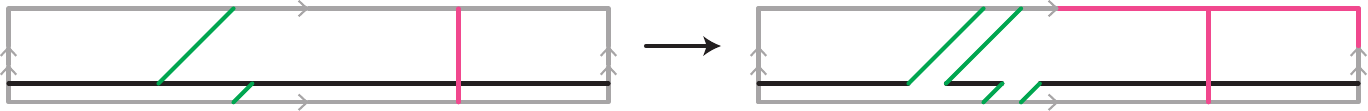}\\
    \caption{How a meridianal boundary compression along a disk $D$ changes $\partial F\subset \partial\nu L$. Color guide for the picture on the left (also in Figures \ref{Fi:Bad_Boundary_Compression1}, \ref{Fi:Bad_Boundary_Compression2}): {\color{black} $\lambda=\partial F$}\color{black}, {\color{myPink} meridian}\color{black}, {\color{myGreen} $\partial D\cap\partial\nu L$} \color{black}}
    \label{Fi:Not_Bad_Boundary_Compression}
\end{figure}

Given a meridianal boundary compression $F\to F'$, orient $L$ and $\partial F\subset \partial \nu L$  coherently, give $\kappa$ the orientation it inherits from $\partial F$, extend this orientation to the meridian $\mu=\kappa\cup\beta$, and write $\varepsilon=-\text{lk}(\mu,L)$. If $\varepsilon=-1$, then $F$ is isotopic to $F'\natural \MobPos$; otherwise, $\varepsilon=1$, and $F$ is isotopic to $F'\natural \MobNeg$. 
    
\begin{rem}\label{R:Ozawa_Rubinstein}
Traditionally, $\partial$-compression of a properly embedded surface $F$ in a manifold $M$ involves surgery along a disk $D$ whose boundary consists of an arc $\alpha\subset F$ and an arc $\beta\subset\partial M$, provided that $\alpha$ is not $\partial$-parallel in $F$.  
In particular, Ozawa and Rubinstein call a properly embedded surface in a link exterior $S^3\setminus\mathring\nu L$ {\it geometrically essential} if it admits neither a geometric compressing disk, as in Definition \ref{D:GeomEss}, nor a $\partial$-compressing disk as defined in this remark \cite{ozawa_rub}.
\end{rem}

The two examples below show, however, that $\partial$-compressing a spanning surface in this way need not yield a spanning surface, thus motivating the more restrictive definitions used throughout this paper.

\begin{figure}[h]
    \centering
    \labellist\small
    \pinlabel {{\color{myGreen} $\alpha_L$}} at 350 150
    \pinlabel {{\color{myRed} $\alpha_F$}} at 410 35
    \endlabellist
    \includegraphics[width=\textwidth]{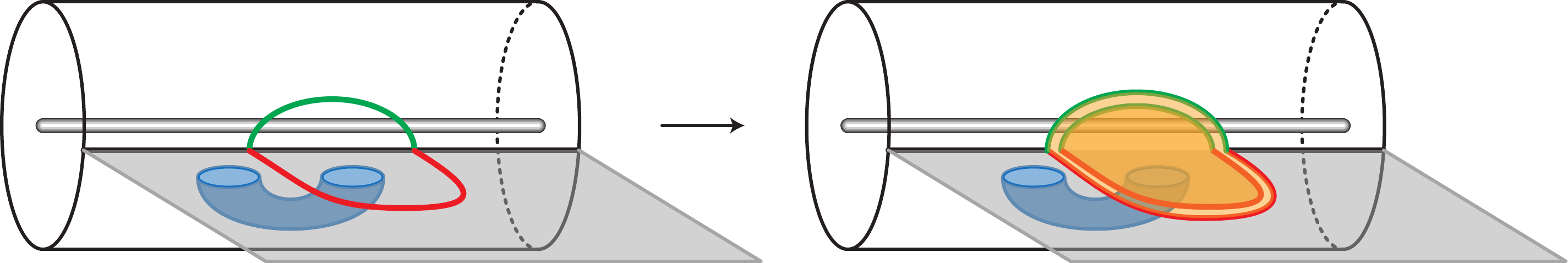}
    \caption{A boundary compression of a spanning surface need not yield a spanning surface. See Example \ref{Ex:Bad_Boundary_Compression}.}
    \label{Fi:Bad_Boundary_Compression}
\end{figure}

\begin{figure}[h]
    \centering
    \includegraphics[width=\textwidth]{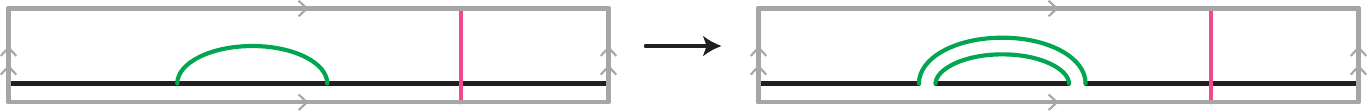}\\
    \caption{How the $\partial$-compression described in Example \ref{Ex:Bad_Boundary_Compression} changes $\partial F=\partial F'\subset \partial\nu L$}
    \label{Fi:Bad_Boundary_Compression1}
\end{figure}

\begin{example}\label{Ex:Bad_Boundary_Compression}
Let $F$ be a spanning surface for a link $L$.
View $F$ as a properly embedded surface in the link exterior $S^3\setminus\mathring{\nu}L$, so that $\partial F=\lambda$ consists of a parallel copy of $L$ on $\partial\nu L$.  See Figure \ref{Fi:Bad_Boundary_Compression}. Take an arc $\alpha\subset \lambda$ and, fixing its endpoints, perturb it slightly (i) in $\partial\nu L$ to get an arc $\alpha_L\subset\partial\nu L$ with $\alpha_L\cap\lambda=\partial\alpha_L=\partial\alpha$ and (ii) in $F$ to get an arc $\alpha_F\subset F$ with $\alpha_F\cap\lambda=\partial\alpha_F=\partial \alpha$.  The arcs $\alpha_L$ and $\alpha_F$ co-bound a disk $D$ whose interior is disjoint from $F$.  Let $D_\alpha$ denote the disk that $\alpha_F$ cuts off of $F$.  Tube $F$ along an arc $\beta$ that is disjoint from $D$ and that has one endpoint in $D_\alpha$ and the other in  $F\setminus D_\alpha$. The disk $D$ is a boundary compression disk for the resulting surface $F'$, but boundary compressing $F'$ along $D$ yields a surface that is not a spanning surface because it has a new boundary component on $\partial\nu L$, namely a trivial circle on $\partial\nu L$ near $\alpha\cup \alpha_L$. See Figures \ref{Fi:Bad_Boundary_Compression} and \ref{Fi:Bad_Boundary_Compression1}.
\end{example}

\begin{example}\label{Ex:Solomon}
Let $L$ denote the right-handed Solomon link $L_{1/4}$, and let $L'$ denote its mirror image $L_{-1/4}$.  
Figure \ref{Fi:1/4} shows an orientation-preserving homeomorphism $\phi:S^3\setminus \mathring\nu L\to S^3\setminus \mathring\nu L'$. 
(There is always an orientation-reversing homeomorphism between the exteriors of links that are mirror images, but the existence of this 
map is more surprising; how unusual is it?)
This map $\phi$ fixes meridia on one link component $L_1\to L'_1$. On the other component $L_2\to L'_2$, however, it takes a meridian $(0,1)$ to a $(-1,1)$ curve, while fixing a longitude $(1,0)$, and thus takes an arbitrary $(1,q)$ curve to a $(q+1,-q)$ curve.  Now take $F$ to be the annulus spanning $L$, whose component-wise boundary slopes are both 2, and take $F^*=F\natural \MobNeg$, where the M\"obius band is attached along $L_2$.  Viewing $F^*$ in the exterior of $L$, its boundary now consists of a $(1,0)$ curve (i.e. a longitude) on $\partial \nu L_1$ and a $(1,2)$ curve on $\partial \nu L_2$.  Write $D$ for a meridianal $\partial$-compressing disk for $F^*$.  

\begin{figure}[h]
\centering
\labellist \small
\pinlabel{$\underset{\text{isotopy}}{\longrightarrow}$} at 320 125
\pinlabel{$\overset{\text{Dehn}}{\underset{\text{twist}}{\longrightarrow}}$} at 725 125
\pinlabel{$\underset{\text{isotopy}}{\longrightarrow}$} at 1130 125
\endlabellist
\includegraphics[width=\textwidth]{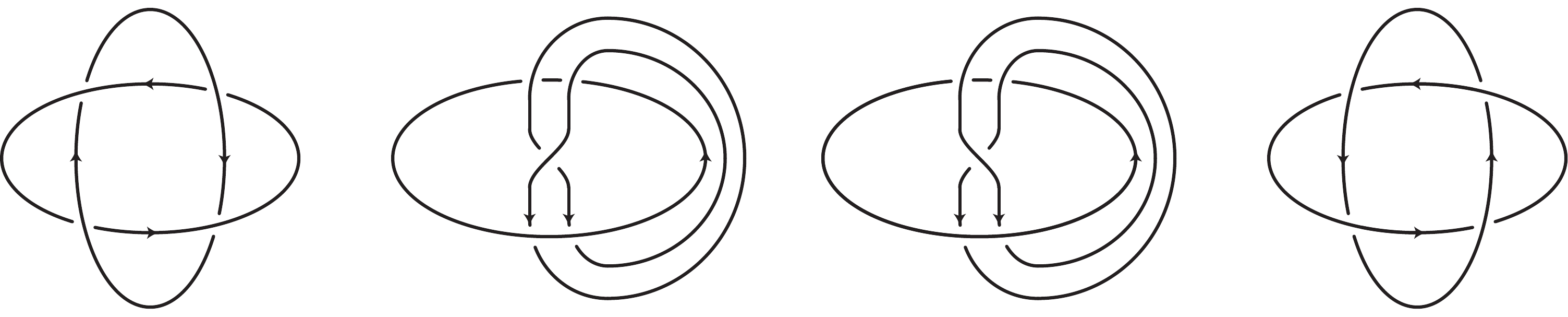}\\
\labellist \tiny
\pinlabel{$\to$} at 375 125
\endlabellist
\includegraphics[width=
3.15in]{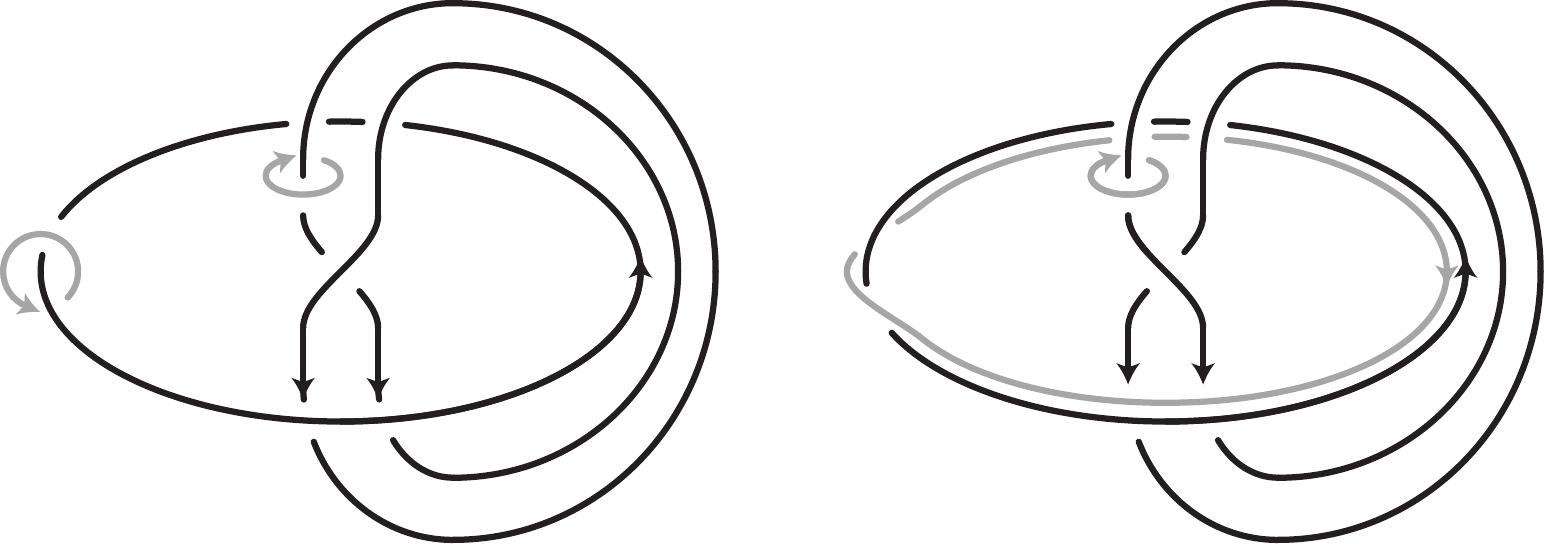}\hfill
\labellist \tiny
\pinlabel{$\to$} at 365 125
\endlabellist
\includegraphics[width=
3.15in]{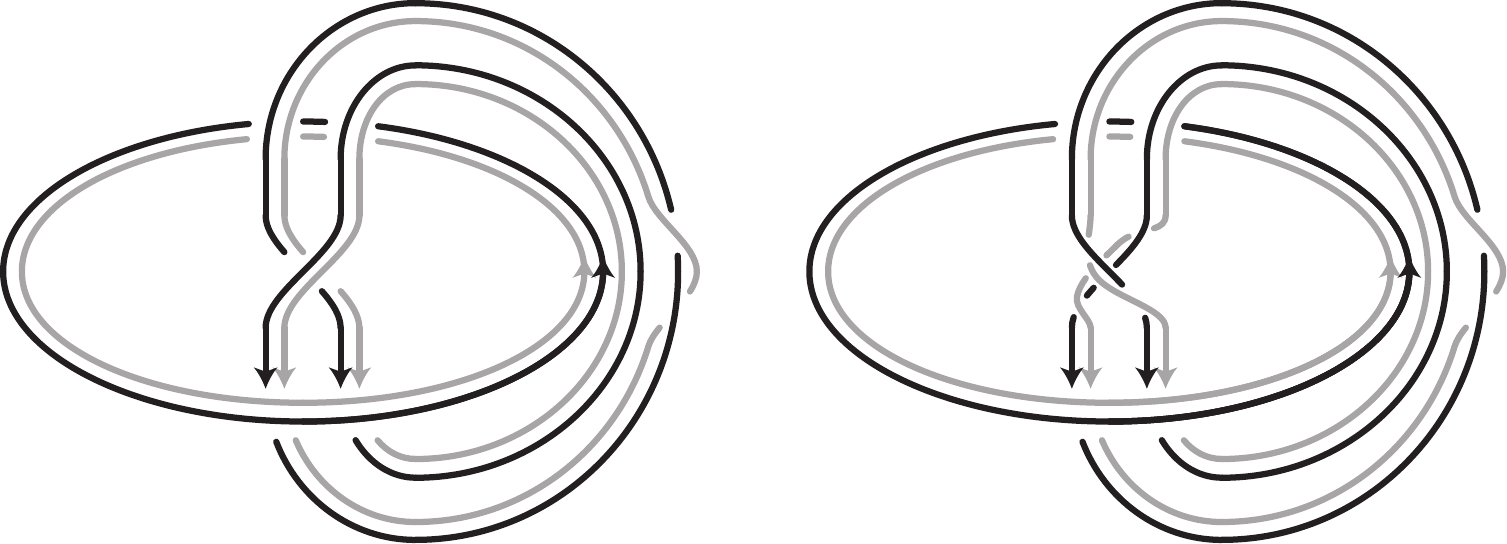}
\caption{Top: A twist along a horizontal disk (not) shown center gives a homeomorphism $\phi:S^3\setminus \mathring\nu L\to S^3\setminus \mathring\nu L'$ for $L= L_{1/4}$ and $L'= L_{-1/4}$. Bottom: Images of meridia and longitudes under $\phi$}
\label{Fi:1/4}
\end{figure}

Write $\phi(F^*)=F'$.  By the preceding discussion, its boundary consists of a $(1,0)$ curve on $\partial \nu L'_2$ and some $(1,q)$ curve on $\partial\nu L'_1$ (because $\partial F^*\cap\partial \nu L_1$ intersects each meridian in a single point and $\phi$ restricts to a homeomorphism of $\partial\nu L_1$ that fixes these meridia).  Therefore, $F'$ is a spanning surface for $L'$.  Moreover, $\phi(D)$ is a $\partial$-compressing disk, in the traditional sense, for $F'$.  Yet, applying this $\partial$-compression to $F'$ yields the surface $\phi(F)$ (because $\partial$-compressing $F^*$ along $D$ yields $F$, and $F'=\phi(F^*)$).  This provides another example of a $\partial$-compression of a spanning surface whose result is not a spanning surface.
\end{example}

We pose the following question, in search of more interesting behavior along the lines of Examples \ref{Ex:Bad_Boundary_Compression} and \ref{Ex:Solomon}.

\begin{figure}[h]
    \centering
    \includegraphics[width=\textwidth]{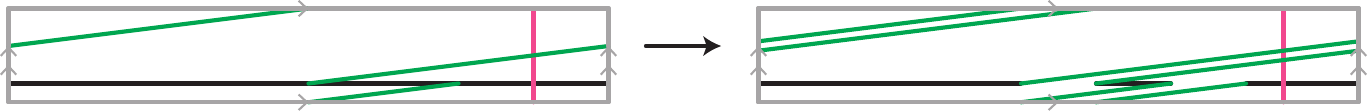}
    \caption{How the boundary compression described in Question \ref{Q:Bad_Boundary_Compression} might change $\partial F\subset \partial\nu L$}
    \label{Fi:Bad_Boundary_Compression2}
\end{figure}

\begin{question}\label{Q:Bad_Boundary_Compression}
    Is there a boundary compression $F\to F'$ of a spanning surface $F$ in a knot exterior $S^3\setminus\mathring{\nu}L$ such that $\partial F'$ is connected, but $F'$ is not a spanning surface? That is, such that $\partial F'$ is a $(p,q)$ curve on $\partial \nu L$ for some $p\neq \pm 1$?  See Figure \ref{Fi:Bad_Boundary_Compression2}.
\end{question}


Although our definition of {\it geometrically essential} differs from that of Ozawa and Rubinstein \cite{ozawa_rub}, these differences entirely reflect a difference of context. Motivated by a variety of conjectures related to the Neuwirth conjecture (see Remark \ref{R:Neuwirth}), Ozawa and Rubinstein were interested in all sorts of properly embedded surfaces in link exteriors, and so it was sensible to consider all compressing disks and $\partial$-compressing disks for such surfaces. In papers like this one specifically concerned with spanning surfaces, however, it is sensible to provide special status to surgery moves that stay within this purview.  

The thematic distinction between {\it algebraic} and {\it geometric} runs throughout this paper and \cite{essence}, and it is valuable to have a linguistic motif to specify this distinction clearly and succinctly.  This is why, for example, we use the term {\it geometrically essential} throughout, rather than the synonymous {\it meridianally essential}---a spanning surface which is not geometrically essential has either a geometric compressing disk or a geometric $\partial$-compressing disk (``geometric'' comes to be implicitly understood because of consistent linguistic choices), whereas a spanning surface which is not meridianally essential has either a geometric compressing disk or a meridianal $\partial$-compressing disk (the conditions are the same, but the cognitive load is greater).  
Nevertheless, we acknowledge the potential for confusion that can arise from using terminology in this sort of flexible way, where certain details in important definitions may differ between papers that are written with distinct contexts.  

Discussions that simultaneously consider both notions of geometric essentiality are trickier in this sense.  While this is likely to be a fairly niche concern, we suggest appropriate language in the course of the next remark.

\begin{rem}\label{R:Neuwirth}
The Neuwirth conjecture asserts that every nontrivial knot $K\subset S^3$ lies on a closed surface $F\subset S^3$ such that $K$ is nonseparating on $F$ and $F\setminus\mathring\nu L$ is $\pi_1$-essential.  While many classes of knots are known to satisfy this conjecture, its status remains open.  

There are several related conjectures, and we highlight three of them here.  Ozawa and Rubinstein posited the ``strong Neuwirth conjecture'' that every non-torus knot $K\subset S^3$ has a $\pi_1$-essential nonorientable spanning surface, but Dunfield found a counterexample in a twisted torus knot which is the closure of a 52-crossing braid on 4-strands. Dunfield's proof uses normal surface theory and computational tools.  It remains open, however, whether every non-torus knot has a nonorientable spanning surface which is geometrically essential in the {\it strong} sense of Ozawa and Rubinstein (see Remark \ref{R:Ozawa_Rubinstein}) or merely {\it as a spanning surface} (as in Definition \ref{D:GeomEss}). Ozawa and Rubinstein call the former the ``weakly strong Neuwirth conjecture" \cite{ozawa_rub}.  Below, we call the latter the ``meridianal Neuwirth conjecture.''
\end{rem}

It is notable, however, that every nontrivial torus knot $K=K_{p,q}$ has a nonorientable spanning surface $F$ which is geometrically essential as a spanning surface.  To see this, assume without loss of generality that $p<q$, and consider the minimal diagram of $K$ which has $(p-1)q$ crossings and $q$-fold rotational symmetry, and denote its all-$A$ and all-$B$ state surfaces.  One of these surfaces, $F_x$, is the unique incompressible Seifert surface for $K$ and has $\beta_1(F_x)=(p-1)(q-1)+1$ with boundary slope $0$.  The other surface $F_y$ satisfies $\beta_1(F_y)=(p-1)q=\beta_1(F_x)+p$, while its slope differs from that of $F_x$ by $2(p-1)q$, which importantly is more than $2p$. Therefore, if one performs meridinal $\partial$-compressions (and, if possible, geometric compressions) on $F_y$ until no more are possible, then the resulting spanning surface $F^*$ will be geometrically essential as a spanning surface but will still have nonzero boundary slope and will therefore be nonorientable.  On the other hand,  Ozawa and Rubinstein prove in \cite[Example 3.11]{ozawa_rub} that this surface $F^*$ is not what they call a {\it pre-essential} when $p$ and $q$ are both odd. In that case, it is not immediately clear whether or not $F^*$ is geometrically essential in their strong sense (if not, this answers Question \ref{Q:Bad_Boundary_Compression}).  With this motivation, we close with our own modified version of the Neuwirth conjecture, which finds further immediate motivation in \cite{kinking}.

\begin{conjecture}[Meridianal Neuwirth conjecture]
Every knot $K\subset S^3$ has a spanning surface which is meridianally essential, i.e. geometrically essential as a spanning surface.
\end{conjecture}

\end{document}